\documentclass[a4paper]{article}
\usepackage{amsmath, amstext, amsxtra, amsthm, amscd, amsgen,amsbsy, amsopn, amsfonts, latexsym, amssymb, amscd, amsthm}
\usepackage{color}

\usepackage{tabularx}

\usepackage{dynkin-diagrams}

\usepackage[a4paper, total={6.5in, 8.5in}]{geometry}

\def\wrt{with respect to}

\def\proofend{\hbox to 1em{\hss}\hfill $\blacksquare $\bigskip }

\newtheorem*{theorem*}{Main Theorem}
\newtheorem*{theoremA}{Theorem A}
\newtheorem*{theoremB}{Theorem B}

\newtheorem{theorem}{Theorem}[section]
\newtheorem{proposition}[theorem]{Proposition}
\newtheorem{lemma}[theorem]{Lemma}
\newtheorem{remark}[theorem]{Remark}

\newtheorem{corollary}[theorem]{Corollary}

\newtheorem{claim}[theorem]{Claim}

\newtheorem{scaltheorem}[theorem]{Vanishing Theorem}

\def\BKRP#1{M_{#1}}
\def\BKRPeps#1#2{M_{(#1,#2)}}

\def\ccc {{\textbf{c}}}
\def\spinc{Spin ^\ccc}

\def\Z{{\mathbb Z}}
\def\R{{\mathbb R}}

\def\C{{\mathbb C}}
\def\N{{\mathbb N}}

\def\ind{\mathrm{ind}}

\def\sign{\mathrm{sign}}

\def\epsi{\epsilon}

\def\Diractwist#1#2{D_{#1,#2}}

\def\DiractwistAPS#1#2{D^+_{#1,#2}}

\def\fdn#1#2{f_{#1,#2}}

\def\so2{S^1}

\def\pullback{pull-back}

\def\groupelement{u}

\begin{document}


\title{Moduli space of metrics of nonnegative sectional or positive Ricci curvature on homotopy real projective spaces}

\author{Anand Dessai, David Gonz\'alez-\'Alvaro}

\date{\today}

\maketitle
\begin{abstract}
We show that the moduli space of metrics of nonnegative sectional curvature on every homotopy $\R P^5$ has infinitely many path components. We also show that in each dimension $4k+1$ there are at least $2^{2k}$ homotopy $\R P^{4k+1}$s of pairwise distinct oriented diffeomorphism type for which the moduli space of metrics of positive Ricci curvature has infinitely many path components. Examples of closed manifolds with finite fundamental group with these properties were known before only in dimensions $4k+3\geq 7$.
\end{abstract}

\section{Introduction}\label{Section: Introduction}
A fundamental problem in Riemannian geometry is to understand how the possible geometries which a given smooth closed manifold can carry is reflected in its topology. Often the geometric property under consideration is expressed in terms of a lower curvature bound for sectional, Ricci or scalar curvature. The topological nature of such a condition is comparably well understood in the case of positive scalar curvature whereas only a few results concerning the relation of lower bounds for sectional or Ricci curvature to topology are known.

Next to the question of whether a metric satisfying a given curvature bound exists, one may ask whether the moduli space of all such metrics, if not empty, has some interesting topology. Again most results in this direction are for metrics of positive scalar curvature whereas much less is known about the topology of the moduli spaces of metrics satisfying lower bounds for sectional or Ricci curvature (for extensive information on the topic see \cite{TW15,T16}).

In the present article we focus on closed manifolds with finite fundamental group and metrics of nonnegative sectional or positive Ricci curvature. Manifolds have been found for which either of the corresponding moduli spaces has infinitely many path components \cite{KS93,KPT05,W11,DKT18,D17,G17}. Besides these results no other topological properties of the respective moduli spaces seem to be known if one restricts to manifolds of dimension $>3$. All these results are confined to manifolds of dimension $4k+3$ and are based on methods related to the Gromov-Lawson relative index \cite[p. 327]{LM89} or use the Kreck-Stolz invariant \cite[Prop. 2.13]{KS93} to distinguish path components of the respective moduli space.

In this paper we consider homotopy real projective spaces of dimension $4k+1\geq 5$ and use eta-invariants to distinguish path components of their moduli spaces. Our first result is for homotopy $\R P^5$s.

\begin{theoremA}
 Let $M$ be a closed manifold homotopy equivalent to $\R P^5$. Then the moduli space $\mathcal{M}_{sec\geq 0}(M)$ of metrics of nonnegative sectional curvature on $M$ has infinitely many path components. The same holds true for the moduli space $\mathcal{M}_{Ric > 0}(M)$ of metrics of positive Ricci curvature on $M$.
\end{theoremA}
It follows that the corresponding spaces of metrics, denoted by $\mathcal{R}_{sec\geq 0}(M)$ and $\mathcal{R}_{Ric > 0}(M)$, respectively, also have infinitely many path components.

There are up to orientation-preserving diffeomorphism precisely four homotopy $\R P^5$s. Each of these can be described in infinitely many ways as a quotient of a Brieskorn sphere of cohomogeneity one \cite{L71}. Grove and Ziller constructed for all of these quotients invariant metrics of nonnegative sectional curvature and asked whether $\mathcal{R}_{sec\geq 0}(M)$ has infinitely many components \cite[Thm. G, Problem 5.6]{GZ00}. Theorem A gives a positive answer to their question for path components. 

Theorem A in conjunction with work of Belegradek, Kwasik and Schultz \cite[Prop. 2.8]{BKS11} shows that for every closed manifold $M$, which is homotopy equivalent to $\R P^5$, the moduli space of complete metrics of nonnegative sectional curvature on the total space of either the trivial or the non-trivial real line bundle over $M$ has infinitely many path components.

In dimension $4k+1>5$ it is an open question which homotopy real projective spaces, besides the standard one, can carry a metric of nonnegative sectional curvature. In each dimension $4k+1$ there are homotopy real projective spaces which can be described, as before, in infinitely many ways as quotients of Brieskorn spheres and which carry metrics of positive Ricci curvature (as first shown by Cheeger \cite{C72}). Our method gives the following result for homotopy $\R P^{4k+1}$s.

\begin{theoremB}
For any $k\geq 1$ there are at least $2^{2k}$ closed manifolds of pairwise distinct oriented diffeomorphism type, which are homotopy equivalent to $\R P^{4k+1}$, and for which the moduli space of metrics of positive Ricci curvature has infinitely many path components.
\end{theoremB}
As before it follows that for each homotopy real projective space $M$ as in Theorem B the space $\mathcal{R}_{Ric > 0}(M)$ has infinitely many path components.

The methods of this paper are used to study path components of moduli spaces but do not give information on other topological invariants like homotopy or homology groups of positive degree (see \cite{TW17} for recent results in this direction for manifolds with {\em infinite} fundamental group). Our approach which utilizes eta-invariants uses in an essential way the fact that the manifolds under consideration are {\em not} simply connected.  This approach can also be used to study the moduli space for quotients of other links of singularities (details will appear in a forthcoming paper). The question whether the results for the homotopy $\R P^{4k+1}$s above also hold for their universal covers remains open.

This paper is structured as follows. In the next section we give an outline of the proofs of Theorem A and Theorem B. In Section \ref{Section: Spaces, operators and eta-invariant} we briefly review spaces and moduli spaces of metrics and describe a relative eta-invariant for $\spinc$-manifolds which will be used to distinguish path components of such spaces of metrics. Section \ref{Section: Review of Brieskorn spheres} is devoted to the construction and the relevant properties of certain Brieskorn spheres. In Section \ref{Section: Cheeger def} we collect basic properties of Cheeger deformations and apply these to the construction of some suitable metrics of positive curvature on Brieskorn spheres and related bordisms. In the last two sections we put this information together and prove Theorem A and Theorem B, respectively.

\bigskip
\noindent
{\em Acknowledgements.} We like to thank Llohann Speran\c{c}a for explaining to us how $G$-$G$ manifolds can be used to construct metrics of positive Ricci curvature, and many thanks also to David Wraith for helpful comments on his work on moduli spaces in \cite{W11}. We are grateful to Uwe Semmelmann, Wilderich Tuschmann, Michael Wiemeler, Fred Wilhelm, David Wraith and the anonymous referees for useful remarks and suggestions.
 This work was supported in part by the SNSF-Project 200021E-172469 and the DFG-Priority programme  {\em Geometry at infinity} (SPP 2026). The second named author received support from the MINECO grant MTM2017- 85934-C3-2-P.


\section{Outline of proofs}\label{Section: Outline}

The goal of this section is to give an outline of the proofs of Theorem A and Theorem B, all the references and details can be found in the following sections.

\smallskip

In the proof of Theorem A we use the following description of homotopy ${\mathbb R} P^5$s. Consider the polynomial $f_d:{\mathbb C} ^4\to {\mathbb C},  (z_1,z_2,z_3,z_4)\mapsto z_1^2+z_2^2+z_3^2+z_4^d$, with $d\geq 3$ odd. Note that $f_d$ has in $0$ an isolated singularity. Its link $\Sigma_0^5 (d)$ is defined as the intersection of $f_d^{-1}(0)$ with the unit sphere. Brieskorn showed that $\Sigma_0^5 (d)$ is a homotopy $5$-sphere, which by the work of Smale is diffeomorphic to $S^5$.

The link $\Sigma_0^5 (d)$ comes with a natural orientation and a cohomogeneity one action by $\so2\times O(3)$ with kernel $\pm(1, Id)$. The elements $\pm (1,-Id)$ act freely by the involution $\tau :(z_1,z_2,z_3,z_4)\mapsto (-z_1,-z_2,-z_3,z_4)$ on $\Sigma_0^5 (d)$. The quotient manifold $\BKRP{d}:=\Sigma_0^5 (d)/\tau$ is a homotopy ${\mathbb R} P^5$. L\'opez de Medrano has shown that each homotopy ${\mathbb R} P^5$ can be obtained in this way. Moreover, $\BKRP{d}$ and $\BKRP{d^\prime}$ are diffeomorphic as oriented manifolds if $d\equiv d^\prime \bmod 16$.

Grove and Ziller constructed on each quotient $\BKRP{d}$ a cohomogeneity one metric $\tilde{g}^d$ of nonnegative sectional curvature. By the work of B\"ohm and Wilking, such a metric evolves immediately to a metric of positive Ricci curvature under the Ricci flow.

In the following we will use the notation $sec\geq 0$, $scal>0$ and $Ric >0$ for nonnegative sectional, positive scalar and positive Ricci curvature, respectively.

We restrict to one homotopy real projective space $\BKRP{d_0} $. By pulling back the metric of $sec\geq 0$ (resp. $Ric >0$) on $\BKRP{d_0+16l}$ to $\BKRP{d_0}$, $l\in \N$, via an orientation preserving diffeomorphism one obtains infinitely many elements in the space of metrics $\mathcal{R}_{sec\geq 0}(\BKRP{d_0})$ (resp. $\mathcal{R}_{Ric > 0}(\BKRP{d_0})$). We will show that these elements represent infinitely many path components of the moduli spaces $\mathcal{M}_{sec\geq 0}(\BKRP{d_0})$ and $\mathcal{M}_{Ric > 0}(\BKRP{d_0})$, respectively. 

\smallskip

The strategy to distinguish path components of $\mathcal{M}_{sec\geq 0}(\BKRP{d_0})$ goes as follows (the argument for $Ric >0$ is similar and easier). For technical reasons we work with the quotient space $\mathcal{R}_{sec\geq 0}(\BKRP{d_0})/ \cal D$, where $\mathcal D$ is the subgroup of diffeomorphisms of $\BKRP{d_0}$ which preserve the topological $\spinc$-structure (which we shall fix below). 

We argue by contradiction: suppose that, for a pair of metrics as above, their classes in $\mathcal{R}_{sec\geq 0}(\BKRP{d_0})/ \cal D$ can be joined by a path. By Ebin's slice theorem such a path lifts to a path in $\mathcal{R}_{sec\geq 0}(\BKRP{d_0})$ between \pullback s of such metrics by elements in $\mathcal D$. Applying the Ricci flow to this path we obtain a new path with the same endpoints and whose interior lies in $\mathcal{R}_{Ric > 0}(\BKRP{d_0})$; in particular the entire path lies in $\mathcal{R}_{scal > 0}(\BKRP{d_0})$, since the Grove-Ziller metrics can be chosen to be of positive scalar curvature. 

In order to arrive at a contradiction we use a suitable relative eta-invariant which, as noted by Atiyah, Patodi and Singer, is constant along paths in $\mathcal{R}_{scal > 0}(\BKRP{d_0})$. This invariant is defined \wrt \ a certain twisted $\spinc$-Dirac operator on $\BKRP{d_0}$, and it is preserved under \pullback s by elements of $\cal D$. We will show that this invariant takes pairwise distinct values for our metrics above. More precisely, we prove that the relative eta-invariant of $(\BKRP{d},\tilde{g}^d)$ is equal to $-d/4$.
 
It follows that $\mathcal{R}_{sec\geq 0}(\BKRP{d_0})/ \cal D$ has infinitely many path components. The same statement holds for the usual moduli space $\mathcal{M}_{sec\geq 0}(\BKRP{d_0})$, since $\mathcal D $ has finite index in the diffeomorphism group $\text{Diff}(\BKRP {d_0})$. 

\smallskip

The computation of the relative eta-invariant of $(\BKRP{d},\tilde{g}^d)$ goes as follows. First, we describe the topological $\spinc$-structure of $\BKRP{d}$. It turns out to be more convenient to move away from the singular fiber $f_d^{-1}(0)$ and to consider instead the fiber $f_d^{-1}(\epsi )$ for some small $\epsi \neq 0$. The intersection of $f_d^{-1}(\epsilon)$ with the unit sphere is again a manifold denoted by $\Sigma_\epsi ^5 (d)$ which is diffeomorphic to $\Sigma_0^5 (d)$; the upshot of this construction is that $\Sigma_\epsi ^5 (d)$ bounds a complex manifold $W_\epsi^6(d)$ without singularities. The group $\mathbb {Z}_{2d}\times O(3)$ still acts on $\Sigma_\epsi ^5 (d)$ and also acts holomorphically on $W_\epsi ^6(d)$. The involution $\tau$ acts freely on $\Sigma_\epsi ^5 (d)$ and the quotient $\BKRPeps{d}{\epsi}:=\Sigma_\epsi ^5 (d)/\tau$ is diffeomorphic to $\BKRP{d}$. The equivariant complex structure on $W_\epsi^6(d)$ induces equivariant $\spinc$-structures on $W_\epsi^6(d)$, $\Sigma_\epsi ^5 (d)$, $\BKRPeps{d}{\epsi}$ and hence a topological $\spinc$-structure on $\BKRP{d}$. 

Second, we describe auxiliary metrics. Following ideas of Cheeger, Lawson, Yau and others we use the action by $\mathbb {Z}_{2d}\times O(3)$ to construct compatible invariant metrics of $scal > 0$ on $\BKRPeps{d}{\epsi}$, $\Sigma_\epsi^5 (d)$ and $W_\epsi^6(d)$ with the following crucial property: the metric on $\BKRPeps{d}{\epsi}$ induces a metric on $\BKRP{d}$ which is in the same path component of $\mathcal{R}_{scal> 0}(\BKRP{d})$ as the Grove-Ziller metric $\tilde{g}^d$. Hence their relative eta-invariants are equal.

The relative eta-invariant of the metric on $\BKRPeps{d}{\epsi}$ is computed in two steps: firstly, in terms of equivariant eta-invariants by lifting the information to its covering $\Sigma_\epsi^5 (d)$, and secondly, by using a Lefschetz fixed point formula for manifolds with boundary in terms of local topological data at the $\tau$-fixed points in $W_\epsi^6(d)$. It turns out that the topological data is the same at every $\tau$-fixed point, and moreover it is independent of $d$. Thus the relative eta-invariant only depends on the number of $\tau$-fixed points in $W_\epsi^6(d)$, which is exactly $d$, and can be computed to be equal to $-d/4$.

\bigskip

The proof of Theorem B is similar in spirit. We consider the link $\Sigma_0^{4k+1}(d)$ of the singularity of ${\mathbb C} ^{2k+2}\to {\mathbb C},  (z_1,\ldots,z_{2k+2})\mapsto z_1^2+\ldots +z_{2k+1}^2+z_{2k+2}^d$ which, as shown by Brieskorn, is a homotopy sphere for $d$ odd. The group $\so2\times O(2k+1)$ acts by cohomogeneity one on $\Sigma_0^{4k+1}(d)$. The involution $\tau :  (z_1,\ldots,z_{2k+1},z_{2k+2})\mapsto (-z_1,\ldots,-z_{2k+1},z_{2k+2})$ acts freely on $\Sigma_0^{4k+1}(d)$ and the quotient, denoted by $M^{4k+1}_d$, is homotopy equivalent to $\R P^{4k+1}$. An inspection of the normal type of $M^{4k+1}_d$ shows the following: the sequence $\{M^{4k+1}_d\}_{d\in\N}$ contains at least $2^{2k}$ distinct oriented diffeomorphism types, all of which are attained by an infinite subsequence. 

Several constructions of metrics of $Ric > 0$ have been carried out for the spaces $M^{4k+1}_d$, starting with work of Cheeger. We use two different approaches, following ideas of Speran\c{c}a and Wraith, respectively, to construct suitable metrics $\tilde{g}^d$ of $Ric > 0$ on $M^{4k+1}_d$. Moreover, the lift of $\tilde{g}^d$ to $\Sigma_0^{4k+1}(d)$ extends to a certain bordism in a way that allows us to compute the relative eta-invariant of $(M^{4k+1}_d,\tilde{g}^d)$ as in Theorem A. It turns out that the latter only depends on $d$ and is equal to $-d/2^{2k}$. 

Let us fix one such manifold $M_{d_0}^{4k+1}$. As stated above, there are infinitely many manifolds $M_{d}^{4k+1}$ which are oriented diffeomorphic to $M_{d_0}^{4k+1}$, thus we can pull back the metrics $\tilde{g}^d$ to $M_{d_0}^{4k+1}$. The relative eta-invariants of these metrics are pairwise distinct. Thus, the same arguments as in the proof of Theorem A show that the metrics represent infinitely many path components of $\mathcal{M}_{Ric > 0}(M_{d_0}^{4k+1})$.


\section{Spaces of metrics, Dirac operators and eta-invariants for $Spin^\ccc$-manifolds}\label{Section: Spaces, operators and eta-invariant}

In this section we first briefly review spaces and moduli spaces of metrics satisfying certain lower curvature bounds (Section \ref{SS: spaces of metrics}). We then describe a relative eta-invariant for $\spinc$-manifolds (Section \ref{SS: APS eta}) which will be used to distinguish path components of such spaces of metrics. We finally recall some tools for computing this invariant (Section \ref{SS: donnelly}).

\subsection{Spaces and moduli spaces of metrics}\label{SS: spaces of metrics}

Let $M$ be a closed manifold, i.e., compact without boundary. We denote by $\mathcal{R}(M)$ the space of all Riemannian metrics on $M$ and endow $\mathcal{R}(M)$ with the smooth topology of uniform convergence of all derivatives. Note that $\mathcal{R}(M)$ is a convex cone, i.e., given $g_1 , g_2\in \mathcal{R}(M)$ and real numbers $a,b>0$, we have $ag_1+bg_2\in \mathcal{R}(M)$. In particular, $\mathcal{R}(M)$ is contractible and locally path-connected.

The diffeomorphism group of $M$, denoted by $\text{Diff}(M)$, acts on $\mathcal{R}(M)$ by pulling back metrics. Note that in general this action is not free and has different orbit types: the isotropy group of this action at a point $g\in\mathcal{R}(M)$ equals the isometry group and this can have a different isomorphism type at different metrics. The orbit space $
\mathcal{M}(M):=\mathcal{R}(M)/\text{Diff}(M)$
equipped with the quotient topology is called the {\em moduli space of metrics} on $M$.

We will need the following two properties of the moduli space. Firstly, being the quotient of a locally path-connected space $\mathcal{M}(M)$ is locally path-connected as well. Secondly, it follows from Ebin's slice theorem \cite{E70} that the quotient map $\mathcal{R}(M)\to \mathcal{M}(M)$ has the path lifting property (see \cite[Prop. VII.6, p. 33]{B75}, \cite[\S 4]{W16} for this classical result).

Let $\mathcal{R}_{scal>0}(M)$, $\mathcal{R}_{Ric>0}(M)$ and $\mathcal{R}_{sec\geq 0}(M)$ denote the subspaces of metrics of positive scalar curvature, positive Ricci curvature and nonnegative sectional curvature, respectively. Since these spaces are invariant under the action of $\text{Diff}(M)$ we can consider the corresponding moduli spaces $\mathcal{M}_{scal>0}(M) =\mathcal{R}_{scal>0}(M)/\text{Diff}(M)$, $\mathcal{M}_{Ric>0}(M) =\mathcal{R}_{Ric>0}(M)/\text{Diff}(M)$ and $\mathcal{M}_{sec\geq 0}(M)  =\mathcal{R}_{sec\geq0}(M)/\text{Diff}(M)$.

Note that for the open subspaces $\mathcal{M}_{scal>0}(M)$ and $\mathcal{M}_{Ric>0}(M)$ of $\mathcal{M}(M)$ every connected component is also path-connected. Note also that the space of metrics under consideration has infinitely many path components if this is the case for its moduli space. For further details on (moduli) spaces of metrics we refer the reader to \cite{E70,TW15}.

\subsection{$Spin^\ccc$-Dirac operators, the APS-index theorem and eta-invariants}\label{SS: APS eta}

In this section we review the construction of an invariant introduced by Atiyah, Patodi and Singer which can be used to distinguish path components of $\mathcal{R}_{scal>0}(M)$ for certain $\spinc$-manifolds $M$. The construction involves Dirac operators, eta-invariants and the Atiyah-Patodi-Singer index theorem.

\subsubsection*{$Spin^\ccc$-Dirac operators}\label{S: spinc structure}

Let us first recall the construction of $\spinc$-structures on manifolds (for details see \cite{ABS64}, \cite[Appendix D]{LM89} and \cite[Appendix D.2]{GGK02}).

Let $(X,g)$ be an $n$-dimensional oriented Riemannian manifold
 and $P_{SO(n)}\to X$ the associated principal bundle of oriented orthonormal frames. Recall that the Lie group $Spin^\ccc  (n):=(Spin(n)\times U(1))/\{\pm (1,1)\}$ is part of the short exact sequence
\begin{equation}\label{EQ: spinc}
1\to \Z_2 \to Spin^\ccc (n) \xrightarrow {\rho \times (\; )^2} SO(n)\times U(1)\to 1
,\end{equation}
where $\rho: Spin(n)\to SO(n)$ denotes the non-trivial double covering.

A $\spinc$-structure on $(X,g)$ consists of a principal $U(1)$-bundle $P_{U(1)}\to X$ and a principal $Spin^\ccc (n)$-bundle $P_{Spin^\ccc (n)}\to X$ together with a $Spin^\ccc (n)$-equivariant bundle map $
P_{Spin^\ccc (n)} \longrightarrow P_{SO(n)} \times P_{U(1)}$. Hence, a $\spinc$-structure is a two-fold covering which restricted to a fiber can be identified (non-canonically) with the covering in \eqref{EQ: spinc}.

For a given principal $U(1)$-bundle, with associated complex line bundle $L$, such a structure exists if and only if $w_2(X)\in H^2(X,\Z_2)$ is the $\mathrm{mod}\, 2$ reduction of the first Chern class $c_1(L)\in H^2(X,\Z)$. In this case the different $\spinc$-structures for the given principal $U(1)$-bundle can be identified (non-canonically) with $H^1(X,\Z_2)$. Note that $P_{U(1)}$ and $L$ are associated to the principal $Spin^\ccc (n)$-bundle since $ P_{Spin^\ccc (n)}/Spin(n) \cong P_{U(1)}$.

In the special case where we consider an oriented Riemannian manifold $(M,g)$ homotopy equivalent to $\R P^{4k+1}$, a $\spinc$-structure exists if and only if $c_1(L)$ is the non-trivial class in $H^2(M,\Z)\cong \Z _2$, i.e., if $P_{U(1)}$ is non-trivial; and there are precisely two different $\spinc$-structures up to equivalence in this case.

The definition above depends on a Riemannian metric on $X$ and a Hermitian metric on $L$ inducing the reduction $U(1)\subset \C ^*$ of its structure group. If one strips away this metric information one obtains the following notion of a {\em topological} $\spinc$-structure which will be used to compare different $\spinc$-structures: Let $P_{GL_n^+(\R )}\to X$ denote the oriented frame bundle of $X$, let $P_{\C ^*}\to X$ be the principal $\C ^*$-bundle associated to $L$, let $\widetilde {GL}_n^+(\R )$ denote the non-trivial double cover of $GL_n^+(\R )$, define $\widetilde G:=(\widetilde {GL}_n^+(\R )\times \C ^*)/\{\pm (1,1)\}$ and consider the exact sequence $1\to \Z_2 \to\widetilde G\to GL_n^+(\R )\times \C ^*\to 1$ corresponding to \eqref{EQ: spinc}.
Then a topological $\spinc$-structure on $X$ for this data consists of a principal $\widetilde G$-bundle $P^{top}_{Spin^\ccc (n)}\to X$ together with a $\widetilde G$-equivariant bundle map $P^{top}_{Spin^\ccc (n)} \longrightarrow P_{GL_n^+(\R )} \times P_{\C ^*}$. We note that in the case of an oriented homotopy $\R P^{4k+1}$ there are precisely two different topological $\spinc$-structures up to equivalence provided $P_{\C ^*}$ is non-trivial.

A $\spinc$-structure induces a topological $\spinc$-structure in view of  the inclusion $\iota :SO(n)\times U(1)\hookrightarrow GL_n^+(\R )\times \C ^*$. Conversely, a topological $\spinc$-structure together with a Riemannian metric on $X$ and a Hermitian metric on $L$ determines a unique $\spinc$-structure. This follows since $\iota$ is a homotopy equivalence. In particular, for $P_{U(1)}$ fixed, a $\spinc$-structure on $(X,g_0)$ determines for any Riemannian metric $g_1$ on $X$ a unique $\spinc$-structure on $(X,g_1)$.

Let $\varphi :X\to X^\prime$ be an isometry between two Riemannian manifolds equipped with $\spinc$-structures and let $\varphi_*:P_{SO(n)}\to P_{SO(n)}^\prime$ denote the induced map between their oriented orthonormal frame bundles. Then the isometry $\varphi $  is said to {\em preserve the $\spinc$-structures} if $\varphi_*$ lifts to an isomorphism of the $\spinc$-structures. Next consider an orientation preserving diffeomorphism $\varphi :X\to X^\prime$ between two oriented manifolds equipped with topological $\spinc$-structures. Let $\varphi_*:P_{GL_n^+(\R )}\to P_{GL_n^+(\R )}^\prime$ denote the induced map between oriented frame bundles. Then the diffeomorphism $\varphi $ is said to {\em preserve the topological $\spinc$-structures} if $\varphi_*$ lifts to an isomorphism of the topological $\spinc$-structures.

For later reference we note that if $X$ is a complex manifold (or, more generally, almost complex manifold) then the complex structure together with a Hermitian metric on $TX$ induces a $\spinc$-structure on $(X,g)$ for which the principal $U(1)$-bundle is associated to the determinant line bundle (see \cite{ABS64}, \cite[p. 392]{LM89}).

For a given $\spinc$-structure one can construct the spinor bundle $S(X)$
 over $X$ (see \cite[D.9]{LM89}). The Levi-Civita connection of $(X,g)$ together with a chosen unitary connection $\nabla ^\ccc $ on $P_{U(1)}\to X$ determine a connection $\nabla$ on $S(X)$. In this situation there is an associated $\spinc$-Dirac operator
$$
D_X:\Gamma (S(X))\to \Gamma (S(X)\otimes T^*X)\to  \Gamma (S(X)\otimes TX)\to \Gamma (S(X)),$$
where the first map is the connection $\nabla$, the second map uses the isomorphism given by the metric $g$ and the last map is induced from Clifford multiplication. Note that the spinor bundle and the associated $\spinc$-Dirac operator are determined by the topological $\spinc$-structure, the Riemannian metric on $X$, the Hermitian metric and the connection $\nabla ^\ccc $.

The situation above can be generalized as follows. Given a Hermitian complex vector bundle $E\to X$ with a Hermitian connection $\nabla ^E$, we can form the tensor product $S(X)\otimes E$ which carries a connection $\hat \nabla $ induced by $\nabla $ and $\nabla ^E$ and can consider the {\em twisted} Dirac operator ${\Diractwist X E}:\Gamma(S(X)\otimes E)\to \Gamma(S(X)\otimes E)$
essentially defined as before.

\subsubsection*{Atiyah-Patodi-Singer index theorem}

Suppose now $(X,g_X)$ is a compact $\spinc$-manifold with non-empty boundary $\partial X=M$ such that the metric $g_X$ is of product form near the boundary $M$, and denote by $g_M$ the induced metric on $M$. We fix a connection $\nabla ^\ccc $ on the associated principal $U(1)$-bundle over $X$ which is constant in the normal direction near the boundary.

Suppose the dimension of $X$ is even. Then $S(X)$ splits as a direct sum $S(X)=S^+(X)\oplus S^-(X)$ and the $\spinc$-Dirac operator on $(X,g_X)$ restricts to an operator ${\mathfrak D}^+:\Gamma (S^+(X))\to \Gamma (S^-(X))$. The restriction of $S^+(X)$  and ${\mathfrak D}^+$ to the boundary $M$ can be identified with the spinor bundle and the $\spinc$-Dirac operator $D_M:\Gamma (S(M))\to \Gamma (S(M))$ on $(M,g_M)$ \wrt \ the induced $\spinc$-structure and connection on the boundary.
Let $P:\Gamma (S^+(X)\vert _{M})\to \Gamma (S^+(X)\vert _{M})$ denote the orthogonal projection onto the space spanned by the eigenfunctions of $D_{M}$ for nonnegative eigenvalues and, following Atiyah, Patodi and Singer, consider the restriction to sections $\phi\in \Gamma(S^+(X))$ with $\phi\vert_{M}\in \ker P$ (\emph{APS-boundary condition}). After imposing the APS-boundary condition the operator ${\mathfrak D}^+$ has finite dimensional kernel and will be denoted by $D_X^+$. Similarly, the formal adjoint of ${\mathfrak D}^+$ (defined via bundle metrics) subject to the adjoint APS-boundary condition has finite dimensional kernel and will be denoted by $(D_X^+)^*$. The index of $D_X^+$ is defined as $\ind \, D_X^+:= \dim\ker \, D_X^+ - \dim\ker \,  (D_X^+)^* \in \Z $ (see \cite{APSI75} for details).

The APS-index theorem expresses the index of $D_X^+$ in terms of geometrical data of $X$ and the operator $D_{M}$. More precisely (\cite{APSI75,APSII75}):
\begin{equation}\label{EQ: APS}
\ind \, D_X^+ = \left ( \int _X e^{\frac{1}{2}c}\hat{{\cal A}}(X,g_X) \right ) - \frac{h(M,g_M) + \eta(M,g_M)}{2}
\end{equation}
We proceed to explain the terms appearing in the right-hand side of \eqref{EQ: APS}. Here $e^x$ equals the power series $\sum _{i\geq 0} \frac{x^i}{i!}$ of a form $x$ and $c=\frac 1 {2\pi } \Omega $ is the first Chern form associated to
$\nabla ^\ccc $.
The term $\hat{{\cal A}}(X,g_X)$ represents the $\hat {\cal A}$-series evaluated on the Pontryagin forms $p_i(X,g_X)$.
 Whereas the integral depends on geometric data on $(X,g_X)$, the remaining two terms depend only on $D_M$.

The term $h(M,g_M)$ denotes the dimension of the kernel of
$D_M$.
The quantity $\eta(M,g_M)$ is the {\em eta-invariant} of $D_M$, which measures the spectral asymmetry of $D_M$.
The eta function is defined by the series
\begin{equation}\label{EQ: eta definition}
\eta(z):=\sum_\lambda \frac{\sign(\lambda)}{\vert\lambda \vert^{z}}, \quad \text{ for } z\in\C ,
\end{equation}
where the sum is over all the non-zero eigenvalues $\lambda$ of $D_M$. It can be shown that this series converges absolutely for real part $\Re (z)\gg 0$ and extends in a unique way to a meromorphic function on $\C$, also denoted by $\eta$, which is holomorphic at $z=0$ \cite[\S 2]{APSI75}. The eta-invariant appearing in formula \eqref{EQ: APS} is just $\eta(M,g_M):=\eta(0)$.

Next we discuss the twisted situation. Let $E\to X$ be a Hermitian complex vector bundle with a Hermitian connection $\nabla ^E$ which is constant in the normal direction near the boundary. Again, for $X$ of even dimension we have a twisted operator $\Gamma(S^+(X)\otimes E)\to \Gamma(S^-(X)\otimes E)$.
Its restriction after imposing APS-boundary conditions will be denoted by ${\DiractwistAPS X E}$. In this situation ${\DiractwistAPS X E}$ has finite index and the APS-index formula \eqref{EQ: APS} takes the form
\begin{equation}\label{EQ: APS twisted}
\ind \, {\DiractwistAPS X E}  = \left ( \int _X ch(E,\nabla ^E) e^{\frac{1}{2}c}\hat{{\cal A}}(X,g_X) \right ) - \frac{h_{E}(M,g_M) + \eta_E(M,g_M)}{2},
\end{equation}
where $ch(E,\nabla ^E)$ is the Chern character for the Chern forms of $(E,\nabla ^E)$ and $h_{E}(M,g_M)$ is the dimension of the kernel of the twisted $\spinc$-Dirac operator ${\Diractwist M {E\vert_M}}$ on $(M,g_M)$. Finally, $\eta_E(M,g_M)$ is the eta-invariant of ${\Diractwist M {E\vert_M}}$. It is defined as in the non-twisted case, but now the sum in \eqref{EQ: eta definition} is over all the non-zero eigenvalues $\lambda$ of ${\Diractwist M {E\vert_M}}$.

\subsubsection*{Dirac operators and positive scalar curvature}

One of the very interesting features of the $\spinc$-Dirac operator of a Riemannian manifold is its relation to the scalar curvature of the metric, as shown by Lichnerowicz. More precisely, the presence of positive scalar curvature is closely related to the injectivity of the operators considered before. For further reference we point out the following well-known crucial consequences for the $\spinc$-Dirac operators $D_M$, $D_X^+$ and their twisted versions (see \cite{L63}, \cite{ASIII68}, \cite[\S 3]{APSII75}, \cite[p. 398, II.8.17 and Appendix D]{LM89}).

\begin{scaltheorem}\label{REMARK: Dirac injective}
\begin{enumerate}
\item Let $(M,g_M)$ be a closed $\spinc$-manifold, $\nabla ^\ccc $ a connection on the associated principal $U(1)$-bundle and $D_M$ its $\spinc$-Dirac operator. If $(M,g_M)$ has $scal>0$ and $\nabla ^\ccc $ has vanishing curvature, then $D_M$ is injective. The same is true if the operator is twisted with a flat bundle $E$. In particular, $h(M,g_M)$ and $h_{E}(M,g_M)$ vanish in this situation.
\item Let $(X,g_X)$ be an even-dimensional compact $\spinc$-manifold with boundary, $g_X$ of product form near the boundary and $\nabla ^\ccc $ a connection of the associated principal $U(1)$-bundle which is constant in the normal direction near the boundary. If $(X,g_X)$ has $scal>0$ and $\nabla ^\ccc $ has vanishing curvature, then $D_X^+ $ and its formal adjoint $(D_X^+)^*$ are injective.  The same is true if the operator is twisted with a flat bundle $E$. In particular, $\ind \, D_X^+$ and $\ind \, {\DiractwistAPS X E}$ vanish in this situation.
\end{enumerate}\qed
\end{scaltheorem}

Recall that a connection on a vector bundle is {\em flat} if its curvature tensor vanishes. A {\em flat vector bundle} is a vector bundle together with a flat connection. It is well known that a complex $m$-dimensional flat vector bundle $E\to X$ must be of the form $E= \hat X\times_{\pi_1(X)} \C^m$, where $\hat X$ denotes the universal cover of $X$ and $\pi_1(X)$ acts diagonally on the product via a linear representation on the $\C^m$ factor (see for example \cite[\S 2]{P81}).

\subsubsection*{An invariant of components of the space of positive scalar curvature metrics}

Let $(M,g_M)$ be a closed connected Riemannian manifold of odd dimension equipped with a $\spinc$-structure, suppose $\nabla ^\ccc $ is a flat connection of the associated principal $U(1)$-bundle and let $\alpha:\pi_1(M)\to U(k)$ be a unitary representation, so that $\alpha$ determines a flat vector bundle (which we denote by $\alpha$ as well). Following Atiyah, Patodi and Singer we define the {\em relative eta-invariant} (also called rho-invariant) as
$$
\widetilde{\eta}_\alpha (M,g_M) := \eta_\alpha (M,g_M) - k\cdot \eta (M, g_M), \qquad\text{ where } k=\text{rank} (\alpha).
$$
This invariant is crucial for our study of the space $\mathcal{R}_{scal>0}(M)$. In fact, as pointed out in \cite[p. 417]{APSII75} invariants like the one above are constant on path components of $\mathcal{R}_{scal>0}(M)$ and, hence, can be applied to distinguish them.
This idea has been used (and refined) to study the space and moduli space of metrics of $scal>0$ for manifolds with non-trivial fundamental group starting with the work of Botvinnik and Gilkey in \cite{BG95,BG96}. For the convenience of the reader we briefly explain the idea in the context of $\spinc$-manifolds. Let us first point out the following application of the APS-index theorem and Theorem \ref{REMARK: Dirac injective} to $scal>0$.

\begin{lemma}\label{LEMMA: twisted eta zero} Let $(M,g_M)$, $\nabla ^\ccc $ and $\alpha $ be as above. Suppose $(X,g_X)$ is a connected compact $\spinc$-manifold with boundary $M$ such that the $\spinc$-structure on $X$ induces the one on $M$, the flat connection $\nabla ^\ccc $ extends to a flat connection on the associated principal $U(1)$-bundle on $X$ which is near the boundary constant in the normal direction, $\alpha$ extends to a representation $\beta: \pi_1(X)\to U(k)$ and $g_X$ is of product form near the boundary and extends $g_M$. If $(X,g_X)$ has $scal>0$ then $\widetilde{\eta}_\alpha (M,g_M)=0$.
\end{lemma}

\noindent
\begin{proof} Consider the $\spinc$-Dirac operator $D_{X}^+$ and the twisted $\spinc$-Dirac operator ${\DiractwistAPS X \beta}$ on $(X,g_X)$. Note that the corresponding operators on the boundary are $D_{M}$ and $\Diractwist M \alpha $, respectively. Comparing the APS-index formulas \eqref{EQ: APS} and \eqref{EQ: APS twisted} for $\ind \, D_{X}^+$ and  ${\DiractwistAPS X \beta}$ one finds that the two integrals only differ by a factor $k$ since $\beta $ is flat.  Hence,
$$\ind \, {\DiractwistAPS X \beta} - k\cdot \ind \, D_{X}^+=- \left (\frac{h_{\alpha }(M,g_M) + \eta_\alpha (M,g_M)}{2}\right ) +k\cdot \left( \frac{h(M,g_M) + \eta(M,g_M)}{2}\right).$$
By Theorem \ref{REMARK: Dirac injective} $
\ind \, D_X^+ , h(M,g_M), \ind \, {\DiractwistAPS X \beta}$ and $h_{\alpha}(M,g_M)$ vanish.
Hence, $\widetilde{\eta}_\alpha (M,g_M)=0$.\end{proof}

Replacing $X$ and its boundary $M$ by $M\times I$ and $M \times \partial I$, respectively, one obtains
\begin{proposition}\label{PROP: eta invariant scal}
Let $M$ be a closed connected Riemannian manifold of odd dimension equipped with a $\spinc$-structure, $\nabla ^\ccc $ a flat connection on the associated principal $U(1)$-bundle and $\alpha:\pi_1(M)\to U(k)$ a unitary representation.
Let  $g_0$ and $g_1$ be two metrics of $scal>0$ which are in the same path component of $\mathcal{R}_{scal>0}(M)$. Then, $\widetilde{\eta}_\alpha (M,g_0)=\widetilde{\eta}_\alpha (M,g_1)$.
\end{proposition}

\noindent
\begin{proof} Let $\gamma(t)$, $t\in [0,1]$, be a continuous path in $\mathcal{R}_{scal>0}(M)$ connecting $g_0$ and $g_1$. We first slightly perturb and rescale the path to obtain a smooth
 path $\hat \gamma (t)$ in $\mathcal{R}_{scal>0}(M)$ which connects $g_0$ and $g_1$ and which is constant near $0$ and $1$.

Consider the manifold $X:= M\times [0,1]$ equipped with the metric $g_X=\hat \gamma(t) + dt^2$. Note that $g_X$ is of product form near the boundary. This metric might not be of $scal>0$ but it is well known that we can deform it (by stretching the interval) to a metric of $scal>0$, also denoted by $g_X$, which restricts to the metrics $\gamma(0)$ and $\gamma(1)$ near the boundary components (see \cite[Lemma 3]{GL80}). 

The boundary $\partial X$ is equal to $M \cup - M$, where $-M$ denotes $M$ with opposite orientation and the induced metric $g_{\partial X}$ on the boundary is just the metric $\gamma(0)$ over $M$ and the metric $\gamma(1)$ over $M$ (with the corresponding orientations).

The $\spinc$-structure on $M$ induces a topological $\spinc$-structure on $M$ which extends canonically to a topological $\spinc$-structure on $X$. The principal $U(1)$-bundle and its flat connection $\nabla ^\ccc $ also extend canonically to $X$. Recall that this data together with the metric $g_X$ define a $\spinc$-structure on $X$ and a flat connection on the associated principal $U(1)$-bundle. The flat vector bundle $\alpha$ over $M$ extends canonically to the flat bundle $\alpha\times [0,1]$ over $X$. Hence, by Lemma \ref{LEMMA: twisted eta zero} we have
$\widetilde{\eta}_\alpha (\partial X,g_{\partial X}) = 0$.
On the other hand, eta-invariants are additive for disjoint unions and are sign-sensitive to orientation, so that $
\widetilde{\eta}_\alpha (\partial X,g_{\partial X}) =\widetilde{\eta}_\alpha (M,\gamma(0))-\widetilde{\eta}_\alpha (M,\gamma(1))$.
The last two identities together yield that $
\widetilde{\eta}_\alpha (M,g_0)=\widetilde{\eta}_\alpha (M,g_1)$.
\end{proof}

\begin{remark}
\normalfont
If $(M,g_M)$ is the boundary of a (possibly simply connected) spin-manifold $(X,g_X)$ of dimension $4k+2$ (and the metric is of product form near the boundary) then the integral in the APS-index formula \eqref{EQ: APS} vanishes for dimensional reasons. In this situation the same argument as above shows that the eta-invariant $\eta (M,g_M)$ of the Dirac operator $D_M$ only depends on the path component of $g_M$ in $\mathcal{R}_{scal>0}(M)$. Note however that $\eta (M,g_M)$ vanishes since the spectrum of $D_M$ is symmetric in this situation (see \cite[p. 61]{APSI75}). A similar remark applies for certain $\spinc$-structures with associated flat $U(1)$-bundle and flat $\alpha$.

\end{remark}

\subsection{Equivariant setting and the computation of the eta-invariant}\label{SS: donnelly}

In general it is very complicated to compute eta-invariants. However, it can be done in some specific situations. Here we discuss the case when we can lift the information to a covering space, and when we have symmetries to compute the invariants locally at the fixed points. Throughout this section $G$ will denote a compact (possibly non-connected or finite) Lie group.

Let $(X,g)$ be an $n$-dimensional $\spinc$-manifold on which $G$ acts by isometries. Suppose the induced action on the orthonormal frame bundle lifts to the $\spinc(n)$-structure $P_{\spinc(n)}\to X$ (in particular, the principal $\spinc(n)$-bundle is $G$-equivariant). Then the associated spinor bundle $S(X)$ is $G$-equivariant and $G$ acts on the space of sections $\Gamma(S(X))$. Suppose the associated $G$-equivariant principal $U(1)$-bundle $P_{U(1)}\to X$ is equipped with a $G$-equivariant unitary connection $\nabla ^\ccc $. Then the latter together with the Levi-Civita connection of $(M,g)$ determines a connection $\nabla$ on the spinor bundle $S(M)$ which is $G$-equivariant. In this situation the $\spinc$-Dirac operator on $(X,g_X)$ is $G$-equivariant.

\subsubsection*{Equivariant eta-invariants and covering spaces}

Let $(M,g)$ be a closed $\spinc$-manifold on which $G$ acts by isometries and suppose that the action lifts to the $\spinc$-structure. In this situation each eigenspace of the $\spinc$-Dirac operator $D_M$ is a finite dimensional $G$-representation. For $\groupelement\in G$ let $\groupelement_\lambda ^\sharp$ denote the action of $\groupelement$ on the eigenspace $E_\lambda$ for the eigenvalue $\lambda $ of $D_M$. 

For fixed $\groupelement\in G$ one can consider the equivariant eta-invariant at $\groupelement$. This invariant, denoted by $\eta(M,g)_\groupelement$, is defined as the value at $z=0$ of the meromorphic extension of the series $\sum_{\lambda \neq 0} \sign(\lambda)\cdot \text{trace} (\groupelement_\lambda ^\sharp)/\vert\lambda \vert^{z}$,
which converges absolutely for $\Re (z)\gg 0$ (see \cite[\S 2]{APS73}, \cite[(2.13) on p. 412]{APSII75}, \cite{D78}). Note that the equivariant eta-invariant at the identity element $1\in G$ is equal to the ordinary eta-invariant, i.e., $\eta(M,g)_1=\eta(M,g)$. Note also that the equivariant eta-invariant $\eta(M,g)_\groupelement$, $\groupelement\in G$, only depends on the equivariant spectrum, i.e., only depends on the representations $E_\lambda $.

Suppose $M$ is connected and has finite fundamental group $G:=\pi _1(M)$. Let $\hat M$ be the universal cover of $M$ and $\alpha :G\to U(m)$ a unitary representation. We denote the associated flat vector bundle $\hat M \times _G \C ^m\to M$ by $\alpha$ as well. Let $\eta _\alpha (M,g)$ be the eta-invariant for the $\spinc$-Dirac operator on $M$ twisted with $\alpha $.

The metric $g$ lifts to a metric $\hat g$ on $\hat M$. The $\spinc$-Dirac operator $D_{M}$ lifts to the $\spinc$-Dirac operator $D_{\hat M}$ on $(\hat M,\hat g)$ which is equivariant \wrt \ the action of $G$ by deck transformations. Let $\eta (\hat M,\hat g)_\groupelement$ denote the equivariant eta-invariant of $D_{\hat M}$ at $\groupelement\in G$.

In this situation one has the following formula expressing the (twisted) eta-invariant $\eta _\alpha (M,g)$ in terms of equivariant (untwisted) eta-invariants $ \eta (\hat M,\hat g)_\groupelement$ (see \cite[Thm. 3.4]{D78}, \cite[p. 390]{BGS97}):
\begin{equation}\label{EQ: Donnelly covering}
\eta _\alpha (M,g)= \frac 1 {\vert G \vert} \left (\sum _{\groupelement\in G} \eta (\hat M,\hat g)_\groupelement\cdot \chi _\alpha (\groupelement)\right ),
\end{equation}
where $\chi _\alpha:G\to\C$ denotes the character of the representation $\alpha$. Whereas the eta-invariant $\eta _\alpha (M,g)$ is usually hard to compute directly, the equivariant eta-invariants $ \eta (\hat M,\hat g)_\groupelement$ can be computed in some situations where $\hat M$ bounds a suitable manifold $W$ in terms of certain local data, as we shall explain in the following section.

\subsubsection*{Donnelly's fixed point formula}

For a closed manifold the link between the equivariant index and local data at the fixed point set is given by the Atiyah-Bott-Segal-Singer Lefschetz fixed point formula. Originally, this formula was proved using methods from $K$-theory. Later Atiyah, Patodi and Singer gave a proof based on the heat equation and extended the index theorem to manifolds with boundary. This method was utilized by Donnelly to prove a fixed point formula for manifolds with boundary. We will discuss this formula in the specific context of equivariant $\spinc$-Dirac operators (see \cite{D78} and references therein for the general discussion).

Let $(W,g_W)$ be a $2n$-dimensional compact $\spinc$-manifold with boundary on which $G$ acts by isometries. Suppose the action lifts to the $\spinc$-structure. We will assume that $g_W$ is of product form near the boundary. Let $g_{\partial W}$ denote the induced metric on the boundary. We will also assume that the principal $U(1)$-bundle on $W$ associated to the $\spinc$-structure is equipped with a $G$-equivariant connection which is constant near the boundary in the normal direction.

Let  $D_{W}^+$ denote the $\spinc$-Dirac operator on $(W,g_W)$ subject to the APS-boundary condition and let $D_{\partial W}$ denote the $\spinc$-Dirac operator on $(\partial W,g_{\partial W})$ \wrt \ the induced $\spinc$-structure and connection. Note that both operators are $G$-equivariant, as well as the formal adjoint $(D_W^+)^*$ of $D_{W}^+$ (defined via invariant bundle metrics). In this situation $\ker \, D_W^+ $ and $\ker \,  (D_W^+)^*$ are $G$-representations and the index of $D_{W}^+$ refines to the virtual complex $G$-representation
$$\ind_G \, D_W^+:= \ker \, D_W^+ - \ker \,  (D_W^+)^* \in R(G),$$
which we identify with its character $G\to \C $, also denoted by $\mathrm{ind}_G \, D_{W}^+$.

For $\groupelement\in G$ we denote by $\ind _G \, D_{W}^+(\groupelement)$ the equivariant index at $\groupelement$, by $h_{\groupelement}$ the character of the $G$-representation $\ker \, D_{\partial W}$ at $\groupelement$ and by $ \eta (\partial W,g_{\partial W})_\groupelement$ the equivariant eta-invariant of $D_{\partial W}$  at $\groupelement$. Then the fixed point formula asserts that
\begin{equation}\label{EQ: LFF}
\ind _G \, D_{W}^+(\groupelement) +\frac 1 2 \left (h_{\groupelement} + \eta (\partial W,g_{\partial W})_\groupelement\right )=\sum _{N\subset W^{\groupelement}}a_N,
\end{equation}
where $a_N$ is a local contribution of a connected component $N$ of the fixed point set $W^{\groupelement}$. Each $a_N$ is an integral over $N$. The integrand depends on the equivariant operator at $N$ and can be expressed in terms of characteristic forms and the action of $\groupelement$ (see \cite[Thm. 1.2]{D78} for details).

We will now assume in addition that $\groupelement\in G$ acts without fixed points on the boundary. In this situation each $N$ is closed. The local contribution $a_N$ is now independent of the metric (see \cite[\S 2]{APS73}) and is equal to the local contribution in the Lefschetz fixed point formula for {\em closed} manifolds. Hence, the contribution can be computed from the equivariant symbol of the $\spinc$-Dirac operator and can be described explicitly in terms of equivariant characteristic classes. Roughly speaking, one takes the cohomological expression in the non-equivariant index theorem, divides out by the Euler class of the normal bundle of $N$ and replaces the formal roots by the equivariant formal roots at $N$ (for a precise statement in the general case see \cite[\S 3]{ASIII68}, for a recipe how to compute the local data see \cite[\S 5.6]{HBJ92}, for a detailed discussion in the case of $\spinc$-structures see for example the appendix of \cite{De96}).

To simplify the discussion we will illustrate this in the special case of an involution $\tau \in G$ which acts with isolated fixed points on $W$ (and no fixed points on $\partial W$). This is the situation which will occur in the study of the moduli space of metrics on homotopy real projective spaces.

The cohomological expression in the non-equivariant index theorem for the $\spinc$-Dirac operator on a closed $2n$-dimensional manifold has the form $\left (\prod _{j=1}^n \frac {x_j}{e^{x_j/2}-e^{-x_j/2}}\right ) \cdot e^{c/2}$,
where $(\pm x_1,\ldots , \pm x_n)$ are the formal roots of the tangent bundle and $c$ denotes the first Chern class of the associated principal $U(1)$-bundle.  Let $pt$ be an isolated fixed point of $\tau$. By the recipe mentioned above the local contribution at $pt$ takes the form $ a_{pt}=\left (\prod _{j=1}^n \frac {1}{(\epsilon_{pt,j}i)-(\epsilon_{pt,j}i)^{-1}}\right ) \cdot \theta _{pt}$,
where $i=\sqrt {-1}$ and the choice of $\epsilon_{pt,j}\in \{\pm 1\}$ and $\theta _{pt}\in \{\pm 1 ,\pm i\}$ depends on the lift of the action to the $\spinc$-structure and on the point $pt$.

In our applications the $\spinc$-structure on $W$ will be induced from a complex structure (resp., almost complex structure) and $\tau$ is holomorphic (resp., is preserving the almost complex structure). In this situation the equivariant symbol of the $\spinc$-Dirac operator is equal to the equivariant symbol of the Dolbeault operator on $W$ (for background information see \cite[\S 3.6]{G95}, \cite[p. 400]{LM89}, \cite[Lemma 5.5]{Du96}, \cite[11.2.4]{N07}) and the local contributions are equal to the corresponding ones for the Dolbeault operator. In particular, the ambiguities in the formula above for the $\spinc$-Dirac operator do not occur. The following computation will be used in Proposition \ref{PROP: eta of special metrics with s nonzero}. 
\begin{proposition}\label{PROP: Dolbeault contribution} Suppose $W$ and $\tau$ satisfy the assumptions given above. Then the local contribution $a_{pt}$ at any isolated $\tau $-fixed point $pt$ is given by
$$a_{pt}=2^{-n}.$$
\end{proposition}
\begin{proof} It suffices to compute the local contribution of the equivariant Dolbeault operator at a fixed point. The cohomological expression in the index theorem for the Dolbeault operator on a closed manifold has the form $\prod _{j=1}^n \frac {x_j}{1-e^{-x_j}}$,
where $(x_1,\ldots ,x_n)$ are the formal complex roots of the tangent bundle.  The equivariant formal complex roots at an isolated fixed point $pt$ of $\tau$ are equal to $(\pi i,\ldots ,\pi i)$. By the recipe mentioned above the local contribution at $pt$ takes the form $a_{pt}=\prod _{j=1}^n \frac {1}{1-e^{-\pi i}}=2^{-n}$.
\end{proof}


\section{Review of Brieskorn spheres}\label{Section: Review of Brieskorn spheres}
In this section we review the construction and properties of certain Brieskorn spheres (see \cite{B66}, further information can be found in the books \cite{HM68,M68}). These manifolds have many nice geometric and topological features. On the one hand they admit large Lie group actions, and in particular free actions by involutions. On the other hand they come with bordisms to which the action extends. Moreover the homotopy real projective spaces which arise as quotients carry interesting metrics of $scal >0$. These properties will be used in the following sections to construct adequate extensions of these metrics to the bordisms and to compute their relative eta-invariants.

\subsection{Brieskorn varieties and Brieskorn spheres}\label{Subsection: Brieskorn varieties and Brieskorn spheres}
Let $\fdn{d}{n+1}:\C^{n+1}\to \C$ be the polynomial defined by 
$$z:=(z_1,\ldots,z_{n},z_{n+1})\mapsto z_1^2+\ldots +z_{n}^2+z_{n+1}^d.$$ For $\epsi\in\C$ consider the hypersurface $V_\epsi^{n}(d) :=\{z\in \C ^{n+1}\, : \, \fdn{d}{n+1}(z)=\epsi\}$.
Then $V_\epsi^{n}(d)$ is an algebraic manifold with no singular points if $\epsi \neq 0$ whereas the Brieskorn variety $V_0^{n}(d)$ has an isolated singular point in $z=0$ if $d\geq 2$.

Let $S^{2n+1}:=\{z\in \C ^{n+1}\, : \, |z_1|^2+\dots+|z_{n+1}|^2=1\}$ and $D^{2n+2}:=\{z\in \C ^{n+1}\, : \, |z_1|^2+\dots+|z_{n+1}|^2\leq 1\}$
denote the sphere and closed disk of radius one in $\R ^{2n+2}=\C ^{n+1}$, respectively. Note that for $\epsi\in\C$ sufficiently small  $V_\epsi^{n}(d)$ and $S^{2n+1}$ intersect transversally \cite[\S 14.3]{HM68}. Let
\begin{align*}
\Sigma _\epsi^{2n-1}(d) & :=V_\epsi^{n} (d)\cap S ^{2n+1} \text{ and}\\
W _\epsi^{2n}(d) & :=V_\epsi^{n} (d)\cap D^{2n+2}.
\end{align*}
Note that $z=0$ is a regular point of $\fdn{1}{n+1}$ and $\Sigma _0^{2n-1}(1)$ is diffeomorphic to the standard unknotted $(2n-1)$-dimensional sphere in $S ^{2n+1}$.
 
Throughout Section \ref{Section: Review of Brieskorn spheres} we will restrict to $d,n$ odd and $d,n\geq 3$.
The manifolds $\Sigma _0^{2n-1}(d)$, i.e., the cases where $\epsi = 0$, have been extensively studied; let us briefly recall the main results. Brieskorn made the crucial observation that $\Sigma _0^{2n-1}(d)$ is a homotopy sphere. By the Generalized Poincar\'e Conjecture, proved by Smale in \cite{S61}, $\Sigma _0^{2n-1}(d)$ is homeomorphic to the standard sphere $S^{2n-1}$. Moreover, $\Sigma _0^{2n-1}(d)$ is diffeomorphic to $S^{2n-1}$ if $d\equiv \pm 1 \mod 8$ and diffeomorphic to the Kervaire sphere if $d\equiv \pm 3 \mod 8$ (see \cite[p. 11]{B66}, \cite[p. 78]{HM68}). The Kervaire sphere is diffeomorphic to the standard sphere in dimensions $5,13,29,61$ and maybe $125$ and is known to be exotic in all other dimensions $2n-1$, $n$ odd (see \cite{HHR06}).

For our purposes we will need the interplay between the manifolds $\Sigma _0^{2n-1}(d)$ and $\Sigma _\epsi^{2n-1}(d)$ with $\epsi\neq 0$. As we shall see in the present section, although they are
 diffeomorphic to each other for small $\epsi$, they have significantly different properties as submanifolds depending on whether $\epsi =0$ or $\epsi\neq 0$.

An important difference is the following: Whereas $\Sigma _\epsi^{2n-1}(d)$ with $\epsi\neq 0$ is the boundary of the complex manifold
$W _\epsi^{2n}(d)$, the Brieskorn sphere $\Sigma _0^{2n-1}(d)$ bounds $W _0^{2n}(d)$ which is only a complex manifold away from its singular point $0\in \C ^{n+1}$. In any case we equip $\Sigma _\epsi^{2n-1}(d)$ and $\Sigma _0^{2n-1}(d)$ with the orientation induced from the complex structure of $W _\epsi^{2n}(d)$ and $W _0^{2n}(d)\setminus \{0\}$, respectively.

The restriction of $\fdn{d}{n+1}:\C ^{n+1}\to \C$ to $S ^{2n+1}$ is a submersion in an open neighborhood ${\cal N}\subset S^{2n+1}$ of $\Sigma _0^{2n-1}(d)$. Hence, the induced map $\mathtt{p}:{\cal N}\to \fdn{d}{n+1}({\cal N})\subset \C$ is a fiber bundle \cite[\S 14.3]{HM68}. In particular, the fiber $\Sigma _\epsi^{2n-1}(d)=\mathtt{p}^{-1}(\epsi)$ is diffeomorphic to $\Sigma _0^{2n-1}(d)=\mathtt{p}^{-1}(0)$ for $\epsi$ sufficiently small (an explicit orientation-preserving diffeomorphism can be constructed from the normalized gradient flow of $\fdn{d}{n+1}$).

Throughout Section \ref{Section: Review of Brieskorn spheres} we will assume that $\epsi$ is a nonnegative real number.

\subsection{Group actions on Brieskorn varieties and their quotients}\label{SS: group actions and quotients}

The constructions above come equipped with certain group actions. Consider the following actions of  $\so2\subset \C$ and $O(n)$ on $\C ^{n+1}$:
$$
\begin{array}{rrll}
A (z_1,\ldots ,z_{n+1})&=&((A(z_1,\ldots ,z_{n})^T)^T, z_{n+1}) &\text{for } A\in\ O(n)\\
w (z_1,\ldots ,z_{n+1})&=&(w^d z_1,\ldots ,w^dz_{n},w^2 z_{n+1}) &\text{for } w\in\ \so2 .
\end{array}
$$
Clearly both actions restrict to actions on $S^{2n+1}$ and $D^{2n+2}$. Note that:
\begin{itemize}
\item the polynomial $\fdn{d}{n+1}$ is $O(n)$-invariant, i.e., $\fdn{d}{n+1}(A z)= \fdn{d}{n+1}(z)$,
\item $\fdn{d}{n+1}( w z)=w^{2d}\fdn{d}{n+1}(z)$, $\forall \, w\in \C $, so that the polynomial $\fdn{d}{n+1}$ is $\Z_{2d}$-invariant, where $\Z_{2d}< \so2$ denotes the set of complex $2d$-roots of unity $\{w\in \C \, : \, w^{2d}=1\}$ and
\item the subspace $\left(\fdn{d}{n+1}\right)^{-1}(0)$ is $\so2$-invariant.
\end{itemize}
We get the following actions on the manifolds defined in Section \ref{Subsection: Brieskorn varieties and Brieskorn spheres}, depending on the value of $\epsi$.

For $\epsi\neq 0$, the manifolds $W_\epsi^{2n}(d)$ and $\Sigma_\epsi^{2n-1}(d)$ admit an action of $\Z _{2d} \times O(n)$ with ineffective kernel $\{\pm (1,Id)\}$. The action on $W_\epsi^{2n}(d)$ is by holomorphisms, i.e., preserves the complex structure of  $W_\epsi^{2n}(d)\subset \C ^{n+1}$. As a consequence, the action on $\Sigma_\epsi^{2n-1}(d)$ is orientation-preserving.

For $\epsi = 0$, the manifolds $W_0^{2n}(d)\setminus \{0\}$ and $\Sigma_0^{2n-1}(d)$  admit an action of $\so2\times O(n)$ with ineffective kernel $\{\pm (1,Id)\}$. Again the action is by holomorphisms on $W_0^{2n}(d)\setminus \{0\}$ and orientation-preserving on $\Sigma_0^{2n-1}(d)$.

The action of $\so2\times O(n)$ on $\Sigma_0^{2n-1}(d)$ is of cohomogeneity one with orbit space an interval. The principal isotropy is $\Z_2 \times O(n - 2)$ and the non-principal isotropies are $S^1\times O(n - 2)$ and $\Z_2 \times O(n - 1)$. In particular, the non-principal orbits (those projecting to the endpoints of the interval) have codimension $2$ and $n-1$, respectively (see for example \cite[p. 212]{BH87} for details).

The action of $O(n)$ on $\Sigma_0^{2n-1}(d)$ is of cohomogeneity two with orbit space a disk (see \cite{HM68}, \cite[p. 54]{B72}). The principal orbits are diffeomorphic to the Stiefel manifold $V_{2,n}(\R )=O(n)/O(n-2)$.

Let $\tau$ denote the action of $\pm (1,-Id)$ on $\C^{n+1}$, i.e.,
$$\tau (z_1,\dots,z_{n+1}) :=(-z_1,\dots,-z_{n},z_{n+1}).$$

For $\epsi < 1$  the involution $\tau$ acts freely on $\Sigma_\epsi^{2n-1}(d)$. The quotient $\Sigma_\epsi^{2n-1}(d)/\tau$ is homotopy equivalent to $\R P^{2n-1}$ (this can be seen by making the classifying map $\Sigma_\epsi^{2n-1}(d)/\tau \to B\Z _2=\R P ^\infty$ cellular and applying the Whitehead theorem). Note that $\Sigma_\epsi^{2n-1}(d)/\tau $ inherits an orientation and an (orientation-preserving) action by $\Z_{2d}\times O(n)$.
For later reference we point out the following result.

\begin{proposition}\label{PROP: quotients by tau are diffeomorphic for small e}
The manifolds $\Sigma_\epsi^{2n-1}(d)/\tau$ and $\Sigma_0^{2n-1}(d)/\tau$ are diffeomorphic as oriented manifolds for sufficiently small $\epsi$.
\end{proposition}

\begin{proof}  
Recall that the involution $\tau $ acts orientation-preservingly on $S^{2n+1}\subset\C^{n+1}$, $\Sigma_0^{2n-1}(d)$ and $\Sigma_\epsi^{2n-1}(d)$. Since the polynomial $\fdn{d}{n+1}$ is $\tau$-invariant it induces a map $\widetilde {\fdn{d}{n+1}}$ from the quotient $S ^{2n+1}/\tau$ to $\C$. Since the restriction of $\fdn{d}{n+1}$ to $S ^{2n+1}$ is a submersion in an open neighborhood of $\Sigma_\epsi^{2n-1}(d)\subset S ^{2n-1}$  for $\epsi$ sufficiently small, the same holds for $\widetilde {\fdn{d}{n+1}}$. Hence, the restriction of $\widetilde {\fdn{d}{n+1}}$ to a suitable small neighborhood $\widetilde{\mathcal{N}}$ of $\Sigma_0^{2n-1}(d)/\tau $ in $S ^{2n+1}/\tau$ is a fiber bundle and $\Sigma_\epsi^{2n-1}(d)/\tau$ is orientation-preservingly diffeomorphic to $\Sigma_0^{2n-1}(d)/\tau$.
\end{proof}

Finally we look at the action of $\tau$ on the bordism $W_\epsi^{2n}(d)$. We get the following properties, which can be checked directly.

\begin{lemma}\label{LEMMA: involution on W} For $\epsi < 1$ and $\epsi\neq 0$ the following hold:
\begin{enumerate}
\item The involution $\tau $ acts on $W_\epsi^{2n}(d)$ with exactly $d$ isolated fixed points $\{p_i\}_{i=1}^d$, all in the interior of $W_\epsi^{2n}(d)$. The fixed points are $p_i=(0,\dots,0,\lambda_i)$, $1\leq i\leq d$,
where $\{\lambda_i\}_{i=1}^d$ denotes the set of complex $d$-roots of the real number $\epsi$.
\item The action of $\Z _{2d}=\Z_{2d}\times \{1\}<\Z_{2d}\times O(n)$ on $W_\epsi^{2n} (d)$ permutes the isolated fixed points.\qed
\end{enumerate}
\end{lemma}

\subsection{Computation of eta-invariants}\label{Subsection: Computation of eta-invariants}
In the proofs of Theorem A and B we will use relative eta-invariants for twisted $\spinc$-Dirac operators to distinguish path components in certain moduli spaces of metrics on
$\Sigma_\epsi^{2n-1}(d)/\tau $. In this section we specify the relevant setting and the computation of the invariant. We begin by describing the $\spinc$-structure.

As remarked in Section \ref{SS: APS eta} the complex structure of $W_\epsi^{2n}(d)$, $\epsi \neq 0$, together with a Hermitian metric on the tangent bundle determines a $\spinc$-structure on $(W_\epsi^{2n}(d),g_W)$ for any metric $g_W$. In the following discussion we will assume that both metrics are $G$-invariant, where $G$ is a closed subgroup of $\Z_{2d}\times O(n)$ containing $(1,-Id)$ or $(-1,Id)$. Then the $\spinc$-structure on $W_\epsi^{2n}(d)$ is $G$-equivariant as well. The induced equivariant $\spinc$-structure on the boundary  $\Sigma_\epsi^{2n-1}(d)=\partial W_\epsi^{2n}(d)$ descends to a $\spinc$-structure on the quotient $\Sigma_\epsi^{2n-1}(d)/\tau$. The latter will be called the {\em preferred $\spinc$-structure} \wrt \ $g_W$ and defines the {\em preferred topological $\spinc$-structure} of $\Sigma _\epsi^{2n-1}(d)/\tau$. This structure will be used in the definition of the relative eta-invariant.

Next we fix a flat connection on the principal $U(1)$-bundle associated to the preferred $\spinc$-structure  of $\Sigma _\epsi^{2n-1}(d)/\tau$ and choose the non-trivial complex line bundle $\alpha :\Sigma _\epsi^{2n-1}(d)\times _{\langle \tau \rangle} \C \to \Sigma _\epsi^{2n-1}(d)/\tau$ as the flat twisting bundle, where $\tau$ acts by $-Id$ on $\C $. Observe that $\alpha$ coincides with the line bundle associated to the principal $U(1)$-bundle, since $n$ is odd.

These data define the relative eta-invariant $ \widetilde{\eta}_\alpha \left(\Sigma_\epsi^{2n-1}(d) /\tau, \tilde{g}\right)$, where $\tilde{g}$ denotes the metric induced by $g_W$.
The next proposition extracts the properties from the situation above which are relevant for the computation of $ \widetilde{\eta}_\alpha \left(\Sigma_\epsi^{2n-1}(d) /\tau, \tilde{g}\right)$ starting from the quotient and determines the invariant under the stated geometric assumptions.

\begin{proposition}\label{PROP: eta of special metrics with s nonzero}
Suppose $0 < \epsi < 1$. Let $\tilde{g}$ be a metric of $scal > 0$ on $\Sigma_\epsi^{2n-1}(d)/\tau$. Then
$$ \widetilde{\eta}_\alpha \left(\Sigma_\epsi^{2n-1}(d) /\tau, \tilde{g}\right)= -2^{-(n-1)}d,$$ provided that the lift $g$ of $\tilde{g}$ to $\Sigma_\epsi^{2n-1}(d)$ extends to a metric $g_W$ on $W_\epsi^{2n}(d)$ which is $\tau$-invariant, of $scal > 0$ and of product form near the boundary.
\end{proposition}

The existence of such a metric will be shown in Corollary \ref{COR: deformation psc 2}. Note that the value of the relative eta-invariant does not depend on the particular choice of the Hermitian metric or the flat connection above.
\begin{proof} In order to lighten notation we will not write the dimensions of the spaces and we will not write the dependence on the metrics $\tilde{g}$ or $g$ for the corresponding eta-invariants. Also we put $\Sigma:=\Sigma_\epsi(d)$ and $W:=W_\epsi(d)$; for example $\widetilde{\eta}_\alpha \left(\Sigma_\epsi^{2n-1}(d) /\tau, \tilde{g}\right)$ will be denoted by $\widetilde{\eta}_\alpha(\Sigma/\tau)$.

The $\spinc$-Dirac operator on $\Sigma/\tau$ associated to the metric $\tilde{g}$ lifts to the $\spinc$-Dirac operator $D_\Sigma$ on $\Sigma$ associated to the metric $g$ (the lift of $\tilde{g}$).
We now use formula \eqref{EQ: Donnelly covering} to compute $ \widetilde{\eta}_\alpha \left(\Sigma/\tau\right)$ in terms of $\Z_2$-equivariant eta-invariants of $D_\Sigma$, where  $\Z_2=\{1,\tau \}$. More precisely we get:
\begin{align*}
\eta _\alpha (\Sigma/\tau) & = \frac 1 2 \left(\eta (\Sigma)_1\cdot \chi_\alpha (1) + \eta (\Sigma)_\tau \cdot \chi_\alpha (\tau)\right)=\frac 1 2 \left(\eta (\Sigma)-\eta (\Sigma)_\tau\right) \text{ and}\\
\eta (\Sigma/\tau) & =\eta _e (\Sigma/\tau) = \frac 1 2 (\eta  (\Sigma)_1\cdot 1 + \eta  (\Sigma)_\tau \cdot 1)=\frac 1 2 (\eta (\Sigma) +\eta  (\Sigma)_\tau ),
\end{align*}
where $e:\pi _1(\Sigma/\tau)\to \Z _2\subset U(1)$ denotes the trivial representation. This gives for the relative eta-invariant 
\begin{equation}\label{EQ: reduced eta 1}
\widetilde{\eta}_\alpha (\Sigma/\tau)=\eta _\alpha (\Sigma/\tau)-1\cdot\eta (\Sigma/\tau)=-\eta (\Sigma)_\tau .
\end{equation}

In order to compute $\eta (\Sigma)_\tau$ we use the Fixed Point Formula \eqref{EQ: LFF} discussed in Section \ref{SS: donnelly}. Since the metrics $g$ and $g_W$ on $\Sigma $ and $W$, respectively, are of $scal > 0$, the $\spinc$-Dirac operator on $\Sigma$ as well as the $\spinc$-Dirac operator on $W$ and its formal adjoint are all injective by Theorem \ref{REMARK: Dirac injective}. Since both metrics are $\tau$-invariant, the corresponding operators are $\tau$-equivariant and their kernels are $\tau$-invariant spaces (and  zero-dimensional by the above). In particular, $\ind _G\,  D_W^+(\tau)$ and $h_{\tau}$ (with the notation as in Section \ref{SS: donnelly}), which are defined in terms of these trivial spaces, both vanish. Recall from Lemma \ref{LEMMA: involution on W} that the set $W^\tau$ consists of exactly $d$ isolated points $\{p_i\}_{i=1}^d$. Thus formula \eqref{EQ: LFF} gives us  $\eta (\Sigma)_\tau =  2 \sum_{i=1}^d  a_{p_i}( \tau )$.

In order to compute the local contributions of the fixed points observe that $W$ and $\tau$ satisfy the assumptions in Proposition \ref{PROP: Dolbeault contribution}. It follows that the local contribution $a_{p_i}(\tau)$ at each fixed point is equal to $2^{-n}$. Hence, by \eqref{EQ: reduced eta 1} we get $ \widetilde{\eta}_\alpha \left(\Sigma/\tau\right)= -2^{-(n-1)}d$.
\end{proof}


\section{Cheeger deformations}\label{Section: Cheeger def}

In this section we recall a method to deform certain metrics that will be used in the proofs of Theorem A and B, and we apply it to construct suitable metrics on the Brieskorn varieties $W_\epsi^{2n}(d)$ defined in Section \ref{Section: Review of Brieskorn spheres}. 

The metrics of interest here come equipped with an isometric Lie group action and the idea is to shrink the metric in direction of the orbits. An early result in this direction is a theorem of Lawson and Yau \cite{LY74} which asserts that a connected manifold with non-trivial $SU(2)$-action admits an invariant metric of $scal>0$. For our purposes it will be more convenient to use a different technique which was used by Cheeger \cite{C72} to construct metrics of $sec \geq 0$ on the connected sum of rank one symmetric spaces. Nowadays this method is known as \emph{Cheeger deformation}. A systematic study of this construction was carried out by M\"uter in his Ph.D. thesis \cite{M87}, see also the concise summary by Ziller \cite{Z09} and \cite[Section 6.1]{AB15} for details. 

\subsection{Review of Cheeger deformations}\label{SS: review of Cheeger def}

Let $M$ be a manifold which at the moment may be non-compact and may have non-empty boundary. Suppose $M$ is equipped with a metric $h$ and suppose a compact Lie group $G$ acts by isometries on $M$. Let $Q$ be a biinvariant metric on $G$. Recall that for a plane $\langle X,Y\rangle $ spanned by orthonormal vectors $X,Y\in {\mathfrak g}=T_eG$ the sectional curvature \wrt \ $Q$ is given by $sec _Q(\langle X,Y\rangle )=\frac 1 4 \vert\vert [X,Y]\vert \vert_Q ^2 \geq 0$.

Consider the product metric $h + \frac{1}{t} Q$ on $M\times G$ for $t>0$ a real number. The action of $G$ on $M\times G$, given by $\mathrm{g_1}\cdot(p,\mathrm{g_2})=(\mathrm{g_1}\cdot p,\mathrm{g_1}\cdot\mathrm{g_2})$, is free and by isometries with respect to $h + \frac{1}{t} Q$. The quotient space $(M\times G)/G$ is diffeomorphic to $M$ and inherits a metric $h_t$ for which the projection
\begin{equation}\label{EQ: riem submersions t}
\pi_t: (M\times G, h + \frac{_1}{^t} Q)  \to (M,h_t),  \quad (p,\groupelement)  \mapsto \groupelement^{-1}p,
\end{equation}
is a Riemannian submersion. The $G$-action on $M$ is by isometries with respect to the submersion metric $h_t$ for all $t>0$. Also, the family of metrics $h_t$ varies smoothly with $t$, and extends smoothly to $t=0$ with $h_0=h$ the original metric (see \cite{M87}).

The importance of the Cheeger deformation $h_t$ comes from the fact that in some cases it ``increases the curvature'' of the original metric. Next we recall the properties relevant for the proof of Theorem A and B.

For $p\in M$, let $G_p$ be the isotropy group of the $G$-action at $p$ and $\mathfrak{g}_p$ its Lie algebra. The $Q$-orthogonal complement $\mathfrak{m}_p$ of $\mathfrak{g}_p$ can be identified via Killing fields  (nowadays also called action fields) with the tangent space $T_p(Gp)$ of the $G$-orbit of $p$, $\mathfrak{m}_p \overset \cong \longrightarrow T_p(Gp)$, $X \mapsto X^*$.
We denote $T_p(Gp)$ by $\mathcal{V}_p$ and define $\mathcal{H}_p$ to be the orthogonal complement of $\mathcal{V}_p$ with respect to $h$. 
This gives a splitting 
\begin{equation}\label{EQ: splitting}
T_p M= \mathcal{V}_p \oplus \mathcal{H}_p=\{ X^* : X\in\mathfrak{m}_p\}\oplus\{v\in T_p M : h(v,X^*)=0,\  \forall X\in\mathfrak{m}_p\}
\end{equation}
which is orthogonal with respect to $h_t$ for every $t\geq 0$.
Next consider the linear map $T_p M \to \mathfrak{m}_p$, $ v \mapsto v_\mathfrak{m}$,
which maps $v$ to the element $v_\mathfrak{m}\in \mathfrak{m}_p$ for which $\left( v_\mathfrak{m}\right)^*$ equals the projection of $v$ onto $\mathcal{V}_p$. Observe that for $X\in \mathfrak{m}_p$ we have $(X^*)_\mathfrak{m} = X$.

Following \cite{M87} consider the automorphisms:
\begin{align*}
P_p: &\, \mathfrak{m}_p \rightarrow \mathfrak{m}_p , & \qquad \text{determined by }& & Q(P_p(X),Y)&=h(X^*,Y^*) & & \forall\, Y\in  \mathfrak{m}_p\\
C_t: &\, T_pM\rightarrow T_pM ,& \qquad \text{determined by }& & h( C_t(X), Y )&= h_t (X,Y) & & \forall\,  Y \in T_pM
\end{align*}
The advantage of these maps is that they can be used to describe, for a given tangent plane $E_0\subset T_pM$, a family of tangent planes $E_t\subset T_pM$ for which the change of sectional curvature $t\mapsto sec _{h_t}(E_t)$ has a nice form and is suitable for computations. In particular, given $v\in T_p M$, the horizontal lift of $C_t^{-1}(v)\in T_p M$ to the point $(p,e)\in M\times G$ with respect to the Riemannian submersion $\pi_t$ (see \eqref{EQ: riem submersions t}) equals $(v,-tP(v_{\mathfrak{m}}))\in T_pM\times T_e G$.

In the remaining part of this section we apply this method to give some lower bounds on sectional and scalar curvatures.

\begin{proposition}\label{PROP: commuting case, bounds on scal} 
Let $v,w\in T_pM$ be linearly independent vectors. Then
\begin{equation}\label{EQ: bound for sec}
sec_{{h}_t}(\langle C_t^{-1}(v),C_t^{-1}(w)\rangle ) \geq \alpha (t) \cdot sec_{h}(\langle v,w\rangle)
\end{equation}
for all $t\geq 0$, where $0<\alpha (t)\leq 1$. If moreover $[P_p v_\mathfrak{m},P_p w_\mathfrak{m}]\neq 0$, then 
$$
\lim_{t\to\infty}sec_{{h}_t}(\langle C_t^{-1}(v),C_t^{-1}(w)\rangle )=\infty .
$$ 
\end{proposition}
 
\noindent
\begin{proof} Let $X:=v_\mathfrak{m}$ and $Y:=w_\mathfrak{m}$. Recall that the horizontal lift at $(p,e)$ of the plane $\langle C_t^{-1}(v),C_t^{-1}(w)\rangle $ in $\pi_t: (M\times G, h + \frac{_1}{^t} Q)  \longrightarrow  (M,h_t)$ is equal to $\langle (v, -t P_pX), (w, -t P_pY)\rangle $.

The sectional curvature of the horizontal lift is of the form
$$sec_{h + \frac{_1}{^t} Q}\langle (v, -t P_pX), (w, -t P_pY)\rangle=\alpha (t) \cdot sec_{h}\langle v,w\rangle + \beta (t)\cdot sec_{\frac 1 t Q}\langle -tP_pX, -tP_pY\rangle $$
\begin{equation}\label{EQ: sectional curvature}
=\alpha (t) \cdot sec_{h}\langle v,w\rangle 
+ t \beta (t)\cdot sec_{Q}\langle P_pX, P_pY\rangle,
\end{equation}
where
$$\alpha (t):= \frac {\vert \vert v\wedge w\vert \vert _{h}^2}{\vert \vert (v, -t P_pX)\wedge (w, -t P_pY)\vert \vert _{h + \frac{_1}{^t} Q}^2}\quad \text{and} \quad \beta (t):=\frac {\vert \vert (-tP_pX)\wedge (-tP_pY)\vert \vert _{\frac 1 t Q}^2}{\vert \vert (v, -t P_pX)\wedge (w, -t P_pY)\vert \vert _{{h + \frac{_1}{^t} Q}}^2}$$
satisfy $0<\alpha (t)\leq 1$, $0<\beta (t)\leq 1$ and $\lim _{t\to \infty} \beta (t)=1$.

The first statement clearly follows from the Gray-O'Neill formula \cite{ON66,G67}. If $P_pX$ and $P_pY$ do not commute then $sec_{Q}\langle P_pX, P_pY\rangle >0$, thus the second term in \eqref{EQ: sectional curvature} goes to infinity for $t\to \infty$ and dominates the first one, which stays bounded. Now the second statement follows from the Gray-O'Neill formula.
\end{proof}

Now we apply these estimates to obtain global lower bounds for the curvature of certain manifolds.

\begin{proposition}\label{PROP: Cheeger sec and scal}
Let $(M,h)$ be a (possibly non-compact) Riemannian manifold. If $sec_{h}\geq 0$ then $sec_{h_t}\geq 0$ for every $t\geq 0$. If in addition one of the following holds
\begin{itemize}
\item $scal_{h}>0$, or 
\item at every point $p\in M$ there exist $X,Y\in \mathfrak m_p$ such that $P_pX, P_pY$ do not commute, 
\end{itemize}
then $scal_{h_t}>0$ for every $t > 0$.
\end{proposition}
\begin{proof}
The fact that $sec_{h_t}\geq 0$ follows from the Gray-O'Neill formula \cite{ON66,G67} for the Riemannian submersion $\pi_t: (M\times G, h + \frac{_1}{^t} Q)  \longrightarrow  (M,h_t)$, since the product metric $h + \frac{_1}{^t} Q$ is of $sec\geq 0$. 

The assumption $scal_{h}>0$ implies that there is a plane of positive sectional curvature for $h$ at every point of $M$. The same is guaranteed for $h_t$ for every $t>0$ thanks to \eqref{EQ: bound for sec}.

The assumption of the existence of $X,Y\in \mathfrak m_p$ with $[P_pX, P_pY]\neq 0$ together with \eqref{EQ: sectional curvature} implies that one has a plane of positive sectional curvature for $h_t$ for every $t>0$ at every point in $M$.
\end{proof}

We will also need the following result.
\begin{proposition}\label{PROP: scal for large t}
Let $K$ be a compact subset of $(M,h)$. Assume that at every point $p\in K$ there exist $X,Y\in \mathfrak m_p$ such that $P_pX, P_pY$ do not commute. Then there exists $t_0\geq 0$ such that the scalar curvature of $(M,h_t)$ is positive on $K$ for every $t > t_0$.
\end{proposition}

\begin{proof} Since $K$ is compact we can find a constant $c\in \R$ such that for every point $p\in K$ and every tangent plane in $T_pM$ the sectional curvature of $(M,h)$ is $\geq c$. By \eqref{EQ: bound for sec} the same is true if one replaces $c$ by $\min (0, c)$ and $h$ by $h_t$, $t>0$. By Proposition \ref{PROP: commuting case, bounds on scal} the assumption on non-commuting elements gives the following: for every point $p\in K$ and every constant $c^\prime \in \R $ there exists an open neighborhood $U$ of $p$ and $t_p>0$ such that for every point $q\in U$ and every $t>t_p$ there is a tangent plane of $(M,h_t)$ at $q$ with sectional curvature $>c^\prime$. Since the scalar curvature is the average of sectional curvatures and $K$ is compact the statement follows.
\end{proof}


\subsection{Deformations on Brieskorn varieties}

In the proofs of Theorems A and B we will need to extend certain metrics on Brieskorn spheres to invariant metrics of $scal > 0$ on some of the bordisms considered in Section \ref{Section: Review of Brieskorn spheres}. This can be done via Cheeger deformations, using the group actions explained in Section \ref{SS: group actions and quotients}. In order to apply the results from Section \ref{SS: review of Cheeger def} we look at the corresponding isotropy groups. 
 
As before let $d,n$ be odd and $d,n\geq 3$. The isotropy groups of the action by $\so2\times O(n)$ on $\Sigma_0^{2n-1}(d)$ are of the form $\Z_2\times O(n-2)$, $\so2\times O(n-2)$ and $\Z_2\times O(n-1)$ \cite[p. 212]{BH87}. It is straightforward to see that one can always find non-commuting vectors in $\mathfrak{m}_p$ for all $p\in  \Sigma_0^{2n-1}(d)$. 

If we restrict to dimension $5$, i.e., $n=3$, and consider invariant metrics of $sec\geq 0$, which exist by the work of Grove and Ziller (see Section \ref{Section: proof Theorem A}), then Proposition \ref{PROP: Cheeger sec and scal} gives us the following

\begin{corollary}\label{COR: deformation psc cohomo 1}
Let $h$ be a metric of $sec\geq 0$ on $\Sigma_0^5(d)$ which is invariant under the action of a closed subgroup $G$ of $\so2\times O(3)$ which contains $O(3)$. Then for every $t>0$ the Cheeger deformation $h_t$ \wrt \ the $G$-action is of $sec\geq 0$ and $scal>0$. The same is true for $\Sigma _0^5(d)/\tau$.\qed
\end{corollary}

Next we consider the $O(n)$-action on the bordism $W_\epsi^{2n}(d)$, $\epsi<1$, and its boundary $\Sigma_\epsi^{2n-1}(d)$. As before let $\tau $ denote the action of $-Id\in O(n)$. The following result guarantees the existence of a metric as in Proposition \ref{PROP: eta of special metrics with s nonzero}.

\begin{corollary}\label{COR: deformation psc 2}
Let $h$ be an $O(n)$-invariant metric on $W_\epsi^{2n}(d)$ such that
\begin{itemize}
\item every fixed point $p_i$ of $\tau$ has an $O(n)$-invariant open neighborhood $U_i$ where $h$ is of positive sectional curvature, and
\item $h$ is of product form near the boundary $\partial W_\epsi^{2n}(d) = \Sigma_\epsi^{2n-1}(d)$.
\end{itemize} 
Then there exists $t_0\geq 0$ such that the Cheeger deformation $h_t$ is of $scal>0$ on $W_\epsi^{2n}(d)$ and of product form near the boundary for every $t\geq t_0$.
\end{corollary}

\begin{proof}
On the open neighborhoods $U_i$ one has that $h_t$ is of $scal>0$ for every $t\geq 0$ by Proposition \ref{PROP: Cheeger sec and scal}.

Let us look at the complement $K:=W_\epsi^{2n}(d)\backslash (\cup U_i)$. We study the isotropy groups of the $O(n)$-action as we did for $\Sigma_0^{2n-1}(d)$ above. There are points $(z_1,\dots,z_{n+1})\in W_\epsi^{2n}(d)$ with only one non-zero coordinate which do not occur in $\Sigma_0^{2n-1}(d)$. We have two kinds:
\begin{itemize}
\item The fixed points $\{p_i\}$ of the $\tau$-action, which do not belong to $K$.

\item Points of the form $(0,\dots, 0, z_i , 0 ,\dots, 0)$, with $i\leq n$ and $z_i^2=\epsi$. The isotropy group at such points is conjugate to $O(n-1)$.
\end{itemize}
The rest of the isotropy groups can be deduced from those of $\Sigma_0^{2n-1}(d)$. As above, one can always find non-commuting vectors in $\mathfrak{m}_p$ for every $p\in K$. Note that $K$ is compact by construction, thus we can apply Proposition \ref{PROP: scal for large t}, which gives the desired result.
\end{proof}

Note that for any given invariant metric on the boundary $\Sigma_\epsi^{2n-1}(d)$ an extension $h$ as in the corollary above always exists: Choose any metric which extends the given metric on the boundary, is of product form near the boundary and invariant of positive sectional curvature near each fixed point (the latter condition can be fulfilled since a neighborhood of a fixed point $p_i$ can be identified equivariantly with a hemisphere of the round sphere $S^{2n}=S(V\oplus \R )$, where $V$ denotes the $O(n)$-representation at $p_i$). After averaging this metric over the $O(n)$-action one obtains a metric $h$ as in Corollary \ref{COR: deformation psc 2}.


\section{Proof of Theorem A}\label{Section: proof Theorem A}

We begin by recalling the following crucial result
 on homotopy $\R P^5$s: there are four oriented diffeomorphism types, all of which are attained by $\Sigma_0^5(d)/\tau$ with the values $d= 1,3,5,7$,
 respectively. Moreover, $\Sigma_0^5(d)/\tau$ and $\Sigma_0^5(d^\prime)/\tau$ are diffeomorphic as oriented manifolds if $d\equiv d^\prime \bmod 16$ (\cite[V.4, V.6.1]{L71}, \cite[p. 512-513]{G69}, \cite{AB68,A77}). The manifolds $\Sigma_0^5(d)/\tau$ are the quotients of Brieskorn spheres discussed in Section \ref{Section: Review of Brieskorn spheres}, where the reader can find all the properties that we will use throughout the present section. Here we denote
$$
\BKRP{d}:=\Sigma_0^5(d)/\tau.
$$
We fix any odd $d$ and consider $\BKRP{d}$, thus we cover all manifolds homotopy equivalent to $\R P^5$. In particular, all the manifolds in the sequence $\{\BKRP{d+16l}\}_{l\in\N}$ are diffeomorphic to each other. Next we proceed to construct the metrics.

Recall that $\Sigma_0^5(d)$ comes with a cohomogeneity one action by $\so2\times O(3)$, which descends to an action on $\BKRP {d}$. The latter is again of cohomogeneity one, with non-principal orbits of codimension $2$. By the work of Grove and Ziller $\BKRP {d}$ admits an $\so2\times O(3)$-invariant metric of $sec\geq 0$, which we will denote by $\tilde{g}^d$ (see \cite[Thm. G]{GZ00}). An inspection of their construction shows that one can assume in addition that $\tilde{g}^d$ is of $scal>0$. Another way to obtain simultaneously $sec\geq 0$ and $scal>0$ is by applying the general deformation argument given in Corollary \ref{COR: deformation psc cohomo 1}.

The main step in the proof is the following computation of the relative eta-invariant of the Grove-Ziller metric $\tilde{g}^d$.

\begin{claim}\label{CLAIM}
Equip $(\BKRP {d},\tilde{g}^d)$ with the preferred $\spinc$-structure (see Section \ref{Subsection: Computation of eta-invariants}) and a flat connection
 on the associated principal $U(1)$-bundle. Let $\alpha$ be the (unique) non-trivial flat complex line bundle over $M_d$. Then the relative eta-invariant of the corresponding $\spinc$-Dirac operator twisted with $\alpha $ is given by
$$\widetilde{\eta}_\alpha \left(\BKRP{d} ,\tilde{g}^d\right)=-\frac{d}{4}.$$
\end{claim}

We postpone the proof of Claim \ref{CLAIM} to the end of the present section and we continue here with the proof of Theorem A.

By the discussion above there exists an orientation preserving diffeomorphism $F_l :\BKRP {d}\to \BKRP {d+16l}$ for every $l\in\N$. We denote by $h_{l}$ the \pullback \ of the Grove-Ziller metric $\tilde{g}^{d+16l}$ on $\BKRP {d+16l}$ to $\BKRP {d}$, i.e.,
$$
h_{l}:=F_l^*\left(\tilde{g}^{d+16l}\right)\in \mathcal{R}_{scal>0}(\BKRP {d})\cap\mathcal{R}_{sec\geq 0}(\BKRP {d}).
$$
The \pullback \ by $F_l$ of the non-trivial complex line bundle over $\BKRP {d+16l}$ is isomorphic to the non-trivial complex line bundle $\alpha$ over $\BKRP {d}$. The \pullback \ by $F_l$ of the preferred $\spinc$-structure on $\BKRP {d+16l}$ induces a topological $\spinc$-structure on $\BKRP {d}$. Recall that $\BKRP {d}$ admits precisely two topological $\spinc$-structures for a fixed orientation (see Section \ref{SS: APS eta}). It follows that infinitely many of the induced structures are equivalent to each other as topological $\spinc$-structures  (and possibly not equivalent to the one induced by the preferred $\spinc$-structure on $\BKRP {d}$). Since relative eta-invariants are preserved via \pullback s, together with Claim \ref{CLAIM} we obtain the following result. 

\begin{claim}\label{CLAIM2}

For each $d$ there exists an infinite subset $\{l_i\}_{i\in\N}\subset \N$ and a topological $\spinc$-structure on $\BKRP {d}$ satisfying the following: For $(\BKRP {d} ,h_{l_i})$, the relative eta-invariant of the corresponding twisted $\spinc$-Dirac operator is given by 
$$
 \widetilde{\eta}_\alpha \left(\BKRP {d},h_{l_i} \right) =\widetilde{\eta}_\alpha \left(\BKRP {d+16l_i}, \tilde{g}^{d+16l_i}\right) = - \frac{d+16l_i}{4}.
$$\qed
\end{claim}

Since the relative eta-invariant is constant on path components of $\mathcal{R}_{scal>0}(\BKRP {d})$, arguing as in \cite[Proposition 2.7]{BKS11} or \cite[Section 2.1]{DKT18} we can conclude the following theorem.

\begin{theorem}\label{THM: aux theorem}
For each $d$, let $\{l_i\}_{i\in\N}$ be the infinite subset from Claim \ref{CLAIM2}. The metrics $\{h_{l_i}\}_{i\in\N}$ represent infinitely many path components of $\mathcal{M}_{sec\geq 0}(\BKRP {d})$ and yield infinitely many path components of $\mathcal{M}_{Ric > 0}(\BKRP {d})$.
\end{theorem}

This implies Theorem A. For the convenience of the reader we recall the details from the arguments in \cite{BKS11,DKT18} with the adequate adaptations.

Let $\mathcal D $ denote the subgroup of diffeomorphisms of $\BKRP {d}$ which preserve its topological $\spinc$-structure. Note that $\mathcal D $ has finite index in the full diffeomorphism group $\text{Diff}(\BKRP {d})$, since $\BKRP {d}$ has only two topological $\spinc$-structures for a fixed orientation. Hence, it suffices to show that the metrics represent infinitely many path components of the quotient space $\mathcal{R}_{sec\geq0}(\BKRP {d})/ \mathcal D$ and yield infinitely many path components of the quotient space $\mathcal{R}_{Ric>0}(\BKRP {d})/ \cal D$ (see also \cite[p. 516]{BG95}).

We argue by contradiction. Suppose there is a path $\tilde{\gamma}:[0,1]\to\mathcal{R}_{sec\geq0}(\BKRP {d})/ \mathcal D$ connecting $[h_{l_i}]$ to $[h_{l_j}]$ with $l_i\neq l_j$. By Ebin's slice theorem (see Section \ref{SS: spaces of metrics}), this path can be lifted to a continuous path $\gamma$ in $\mathcal{R}_{sec\geq 0}(\BKRP {d})$ with $\gamma(0)=h_{l_i}$ and $\gamma(1)=\Phi^*(h_{l_j})$ for some $\Phi\in\cal D$.

By the work of B\"ohm and Wilking in \cite{BW07}, a metric on $M_d$ of $sec\geq 0$ evolves under the Ricci flow immediately to a metric of $Ric>0$. Thus, by applying the Ricci flow to every metric in the path $\gamma$ yields a path of metrics of $Ric>0$. Concatenation of this path and the trajectories of the endpoints of $\gamma$ gives a path $\hat{\gamma}$ with the same endpoints as $\gamma$ and whose interior-points lie in $\mathcal{R}_{Ric>0}(\BKRP {d})$; in particular the entire path $\hat{\gamma}$ lies in $\mathcal{R}_{scal>0}(\BKRP {d})$. 

By Claim \ref{CLAIM2} there is a topological $\spinc$-structure on $M_d$ for which the relative eta-invariant of the endpoints of $\hat{\gamma}$ is given by
\begin{equation}\label{EQ: etas endpoints}
\widetilde{\eta}_\alpha \left(\BKRP {d},h_{l_i} \right) = - \frac{d+16l_i}{4}, \qquad \widetilde{\eta}_\alpha \left(\BKRP {d},\Phi^*(h_{l_j}) \right) = - \frac{d+16l_j}{4},
\end{equation}
where the second identity follows from the fact that the relative eta-invariant is preserved under \pullback . Since the quantities in \eqref{EQ: etas endpoints} are different and the relative eta-invariant is constant on path components of $\mathcal{R}_{scal>0}(\BKRP {d})$ (see Proposition \ref{PROP: eta invariant scal}) we get a contradiction.

Hence the classes $\{[h_{l_i}]\}_{i\in\N}$ belong to pairwise different path components of $\mathcal{R}_{sec\geq0}(\BKRP {d})/ \mathcal D$, and moreover the proof shows that the deformed metrics of $Ric > 0$ belong to different path components of $\mathcal{R}_{Ric>0}(\BKRP {d})/ \cal D$. By the discussion above this completes the proof.\qed

\begin{proof}[Proof of Claim \ref{CLAIM}] 

In the proof we will drop the dimension from the notation and we set $\BKRPeps {d} {\epsi} := \Sigma _{\epsi }(d)/\tau$. Hence, $\BKRPeps {d} {0}=\BKRP {d}$.  The basic idea is to use Cheeger deformations and apply the computation for eta-invariants on $\BKRPeps {d} {\epsi} $ of Section \ref{Subsection: Computation of eta-invariants}. For $\epsi _0>0$ define 
$$
\tilde{Z}(d):=\tilde f_d^{-1}([0,\epsi _0])=\overset {\cdot}\bigcup_{0\leq \epsi \leq \epsi _0} \BKRPeps {d} {\epsi}.
$$
Recall that $O(3)$ acts on $
\tilde{Z}(d)$ and each subspace $\BKRPeps {d} {\epsi} $. Recall also that the universal cover of $\BKRPeps {d} {\epsi _0}$ is the boundary of the $O(3)$-manifold $W _{\epsi _0}(d)$. We construct the following invariant metrics on them:

\bigskip
\noindent
\begin{tabularx}{\linewidth}{ r X } $(\BKRPeps {d} {0},\tilde{g}^d)$ & This is just the Grove-Ziller metric described at the beginning of the present section, which is in particular $O(3)$-invariant.\\
$(\tilde{Z}(d), \tilde{h})$ & Extend the metric $\tilde{g}^d$ to a metric $\tilde{h}$ on $\tilde{Z}(d)$. After averaging the metric over $O(3)$, we can assume that $\tilde h$ is in addition $O(3)$-invariant.\\
$(W _{\epsi _0}(d),k)$ & Consider the restriction of $\tilde{h}$ to $\BKRPeps {d}{\epsi _0} $ and lift it to $\Sigma _{\epsi _0}(d)$. The latter is an $O(3)$-invariant metric and we can extend it to a metric $k$ on $W _{\epsi _0}(d)$ such that $k$ is of product form near the boundary $\Sigma _{\epsi _0}(d)$. Moreover we may choose the metric $k$ such that its restriction to a sufficiently small $O(3)$-invariant neighborhood of each $\tau$-fixed point is an $O(3)$-invariant metric of positive sectional curvature (see the remark after Corollary \ref{COR: deformation psc 2}). After averaging we can assume that $k$ is in addition $O(3)$-invariant on all of $W _{\epsi _0}(d)$.
\end{tabularx}

\bigskip
\noindent
Now we use the $O(3)$-actions to produce Cheeger deformations $\tilde{g}_t^d$, $\tilde{h}_t$ and $k_t$ of the metrics above, where $t \geq 0$ denotes the parameter of the deformation. By the results in Section \ref{Section: Cheeger def} there exists $t_0\geq 0$ such that the following properties hold:
\begin{enumerate}
\item[(P1)] The restriction of $\tilde{h}_t$ to $\BKRPeps{d}{\epsi }$  has $scal>0$ for every $t\geq t_0$ and for every $0 \leq \epsi \leq \epsi_0$ (observe that the restriction of $\tilde{h}_t$ to $\BKRPeps{d}{\epsi }$ equals the Cheeger deformation of the restriction of $\tilde{h}$ to $\BKRPeps{d}{\epsi }$ at time $t$). This follows by Proposition \ref{PROP: scal for large t}.

\item[(P2)] The metric $k_t$ on $W _{\epsi _0}(d)$ has $scal >0$ and is of product form near the boundary for every $t\geq t_0$. This follows by Corollary \ref{COR: deformation psc 2}. Observe that the metric $k_t$ restricted to the boundary $\Sigma _{\epsi _0}(d)$ descends to a metric on $\BKRPeps{d}{\epsi _0}$ which agrees with the metric $\tilde{h}_t$ restricted to $\BKRPeps{d}{\epsi _0}$.
\end{enumerate}
Now we construct a path $\tilde{\gamma}$ in $\mathcal{R}_{scal>0}(\BKRPeps{d}{\epsi _0})$ as follows. Recall from the proof of Proposition \ref{PROP: quotients by tau are diffeomorphic for small e} that we can fix a trivialization $\BKRPeps{d}{\epsi_0}\times [0,\epsi_0]\overset \cong \longrightarrow \tilde Z(d)$
 which extends the identity $\BKRPeps{d}{\epsi_0}\times \{\epsi_0\}\longrightarrow \tilde Z(d)\vert _{\BKRPeps{d}{\epsi_0}}$  (such a trivialization can also be constructed using a normalized gradient flow). Let us denote the associated family of diffeomorphisms by
$$
\varphi _\epsi : \BKRPeps{d}{\epsi_0}\to \BKRPeps{d}{\epsi},\qquad \epsi\in [0,\epsi_0].$$
Since $\varphi _{\epsi _0}$ is the identity on $\BKRPeps{d}{\epsi_0}$ by continuity the diffeomorphisms $\varphi _\epsi$, $\epsi\in [0,\epsi_0]$, preserve the preferred topological $\spinc$-structures. We use the Cheeger deformation of $\tilde{g}^d$ and the family $\varphi _\epsi $ to construct two paths of metrics on $\BKRPeps{d}{\epsi _0}$:
$$
\begin{array}{rclccrcl}
\tilde{\gamma}_1:[0, t_0] & \longrightarrow & \mathcal{R}_{scal>0}(\BKRPeps{d}{\epsi _0}), & &  & \tilde{\gamma}_2:[0, \epsi_0] & \longrightarrow & \mathcal{R}_{scal>0}(\BKRPeps{d}{\epsi _0})\\
t & \longmapsto & \varphi _0^* \left(\tilde{g}_{t}^d\right) & & & \epsi & \longmapsto & \varphi _\epsi^* \left(\tilde{h}_{t_0}\vert_{\BKRPeps{d}{\epsi }} \right)
\end{array}
$$
Observe that the metrics in $\tilde{\gamma}_1$ are of $scal>0$ by Corollary \ref{COR: deformation psc cohomo 1}, while the metrics in $\tilde{\gamma}_2$ are of $scal>0$ by (P1) above. Note that $\tilde{\gamma}_1(t_0)= \varphi _0^* (\tilde{g}_{t_0}^d)=\varphi _0^* (\tilde{h}_{t_0}\vert_{\BKRPeps{d}{0}})=\tilde{\gamma}_2(0)$.

The concatenation of the paths $\tilde{\gamma}_1$ and $\tilde{\gamma}_2$ determines a path $\tilde{\gamma}$ in $\mathcal{R}_{scal>0}(\BKRPeps{d}{\epsi _0})$, whose endpoints are $\varphi _0^*(\tilde g^d)$ and $\tilde{h}_{t_0}\vert_{\BKRPeps{d}{\epsi _0}}$.

We finish with the computation of $\widetilde{\eta}_\alpha \left(\BKRP{d} ,\tilde{g}^d\right)$. Recall that the diffeomorphisms $\varphi _\epsi$ preserve the preferred topological $\spinc$-structures, so that from now on relative eta-invariants on $\BKRPeps{d}{\epsi }$ are understood to be defined \wrt \ the preferred topological $\spinc$-structures, the flat connection
 on the associated principal $U(1)$-bundle and the non-trivial flat complex line bundle $\alpha $.

Note that $\widetilde{\eta}_\alpha \left(\BKRP{d} ,\tilde{g}^d\right)$ equals $\widetilde{\eta}_\alpha \left(\BKRPeps{d}{\epsi _0},\varphi _0^*(\tilde g^d)\right)$ because the relative eta-invariant is preserved under \pullback s. Since the path $\tilde{\gamma}$ constructed above lies in $\mathcal{R}_{scal>0}(\BKRPeps{d}{\epsi _0})$, it follows by Proposition \ref{PROP: eta invariant scal} that moreover $\widetilde{\eta}_\alpha \left(\BKRPeps{d}{\epsi _0},\varphi _0^*(\tilde g^d)\right)$ equals $\widetilde{\eta}_\alpha \left(\BKRPeps{d}{\epsi _0},\tilde{h}_{t_0}\vert_{\BKRPeps{d}{\epsi _0}}\right)$. 

Now (P2) above is precisely telling us that the metric $\tilde{h}_{t_0}\vert_{\BKRPeps{d}{\epsi _0}}$ on $\BKRPeps{d}{\epsi _0}$ satisfies the assumptions in Proposition \ref{PROP: eta of special metrics with s nonzero}, which gives us the value of the relative eta-invariant. The discussion can be summarized in the following way:
$$\widetilde{\eta}_\alpha \left(\BKRP{d} ,\tilde{g}^d\right)  =\widetilde{\eta}_\alpha \left(\BKRPeps{d}{\epsi _0},\varphi _0^*(\tilde g^d)\right)=\widetilde{\eta}_\alpha \left(\BKRPeps{d}{\epsi _0},\varphi _0^*(\tilde g^d_{t_0})\right)=\widetilde{\eta}_\alpha \left(\BKRPeps{d}{\epsi _0},\tilde{h}_{t_0}\vert_{\BKRPeps{d}{\epsi _0}}\right)=-\frac d {4}.$$
\end{proof}

Note that in the proof of Claim \ref{CLAIM} we are only using the dimension restriction and the fact that $\tilde{g}^d$ is of $sec\geq 0$ to ensure that the Cheeger deformation $\tilde{g}^d_t$ is of $scal > 0$ for every $t\geq 0$. If one assumes this property for large $t$ from the beginning the same proof gives the following general result.

\begin{claim}\label{CLAIM: computation of eta-inv for Cheeger scal>0} Let $d,n$ odd and $d,n\geq 3$. Equip $M^{2n-1}_d:=\Sigma_0^{2n-1}(d)/\tau$ with the preferred $\spinc$-structure, a flat connection on the associated principal $U(1)$-bundle and let $\alpha$ be the (unique) non-trivial flat complex line bundle over $M^{2n-1}_d$. Let $\tilde{g}$ be an $O(n)$-invariant metric satisfying the following: there exists $t_0\geq 0$ such that the Cheeger deformation $\tilde{g}_t$ is of $scal > 0$ for every $t\geq t_0$.
 Then for the metric $\tilde{g}_{t_0}$, the relative eta-invariant of the corresponding $\spinc$-Dirac operator twisted with $\alpha $ is equal to
$$\widetilde{\eta}_\alpha \left(M_d^{2n-1} ,\tilde{g}_{t_0}\right)=-2^{-(n-1)}d.$$
\qed
\end{claim}


\section{Proof of Theorem B}\label{Section: proof Theorem B}
As in the proof of Theorem A we will use relative eta-invariants to distinguish path components of the moduli space. Again, the manifolds involved in the proof are the quotients of Brieskorn spheres $\Sigma_0^{2n-1}(d)/\tau$ with $d,n$ odd and $n\geq 3$ as discussed in Section \ref{Section: Review of Brieskorn spheres}. These quotients are homotopy equivalent to a real projective space of dimension $4k+1=2n-1$. We will also consider the regular case $d=1$ for which $\Sigma_0^{4k+1}(1)/\tau$ can be identified directly with $\R P^{4k+1}=S^{4k+1}/\{\pm Id\}$. Let
$$
M^{4k+1}_d:=\Sigma_0^{4k+1}(d)/\tau. 
$$
As shown by L\'opez de Medrano there are only finitely many smooth homotopy $\R P^{4k+1}$s up to orientation-preserving diffeomorphisms \cite[\S IV.3.4 and Ch. V]{L71}. Hence, the manifolds $M^{4k+1}_d$ with $d$ odd belong to finitely many oriented diffeomorphism types. From now on we will always assume that all maps are orientation-preserving if not otherwise stated.

Atiyah and Bott \cite[p. 489]{AB68},  \cite{A77} used equivariant index theory to prove that if two Brieskorn quotients $M^{4k+1}_d$ and $M^{4k+1}_{d^\prime}$ are diffeomorphic then $d\equiv \pm d^\prime \bmod 2^{2k+2}$ (prior to \cite{A77} Giffen \cite{G69} has shown that there are at least $2^{2k}$ pairwise non-diffeomorphic Brieskorn quotients $M^{4k+1}_d$ distinguished by their smooth normal invariant, which only depends on $d \bmod 2^{2k+2}$).

For each integer $k\geq 1$ and each odd $d_0$, $0<d_0<2^{2k+1}$, there exists an infinite sequence $\{l_i\}_{i\in \N}$ of pairwise distinct integers $l_i\in 2^{2k+2}\N $ such that all the manifolds in the sequence $\{M^{4k+1}_{d_0+l_i}\}_{i\in\N}$ are diffeomorphic to each other. To prove Theorem B it suffices to show that for any such $d_0$ the moduli space $\mathcal{M}_{Ric > 0}(M^{4k+1}_{d_0})$ has infinitely many path components.

It has been shown by various methods that the Brieskorn spheres and their quotients $M^{4k+1}_d$, $d$ odd, admit metrics of positive Ricci curvature. The first construction we know of is due to Cheeger who used the cohomogeneity one action on Brieskorn spheres to exhibit metrics of $Ric\geq 0$. He also pointed out that in combination with deformation techniques of Aubin and Ehrlich this yields a metric of $Ric > 0$ which is invariant under the action of $\so2\times O(2k+1)$ and hence descends to $M^{4k+1}_d$ (see \cite{C72},
see also \cite{H75}).  It was later shown by Grove and Ziller \cite{GZ02} as well as Schwachh\"ofer and Tuschmann \cite{ST04} that every manifold with finite fundamental group and a cohomogeneity one action admits an invariant metric of $Ric >0$ (see \cite{GVWZ06} for obstructions to the existence of invariant metrics of $sec\geq 0$).

In \cite{W97} Wraith used surgery to construct metrics of $Ric> 0$ on Brieskorn spheres and, more generally, on all homotopy spheres which bound parallelizable manifolds. Boyer, Galicki and Nakamaye gave in \cite{BGN03} a new proof of the existence of metrics of $Ric> 0$ on these homotopy spheres based on Sasakian geometry for links of isolated singularities.  Another more recent approach which we learned from Speran\c{c}a uses Cheeger deformations for $G$-$G$ bundles.

These constructions (after slight modification if necessary) yield metrics of positive Ricci curvature on the quotients $M^{4k+1}_d$. It would be interesting to understand whether these metrics all induce elements in the same path component of the moduli space $\mathcal{M}_{Ric > 0}(M^{4k+1}_{d})$. It is likely that Proposition \ref{PROP: eta of special metrics with s nonzero} can be used to show that at least all these metrics have the same relative eta-invariants. We will verify this in two instances and thereby prove the following

\begin{proposition}\label{PROP: eta-inv of Ric metrics}
The manifold $M^{4k+1}_d$ admits a metric $\tilde{g}^d$ of $Ric >0$ for which the following holds. Equip $M^{4k+1}_d$ with the preferred $\spinc$-structure (see Section \ref{Subsection: Computation of eta-invariants}), fix a flat connection
 on the associated principal $U(1)$-bundle and let $\alpha$ be the (unique) non-trivial flat complex line bundle over $M^{4k+1}_d$. Then the relative eta-invariant of the corresponding twisted $\spinc$-Dirac operator is equal to 
$$\widetilde{\eta}_\alpha \left(M^{4k+1}_d ,\tilde{g}^d\right)=-2^{-2k}d.$$
\end{proposition}

As mentioned above, we give two independent proofs for this proposition below, one using Speran\c{c}a's approach via $G$-$G$ bundles, the other one using a construction of Wraith. We postpone the corresponding proofs to Section  \ref{SS: Speranca approach} and \ref{SS: Wraith approach}, respectively, and we continue here with the proof of Theorem B.

\smallskip

Recall that there exists a diffeomorphism
 $F_{l_i} :M^{4k+1}_{d_0}\to M^{4k+1}_{d_0+l_i}$ for every $i\in\N$. We denote by $h_{l_i}$ the \pullback \ of the metric $\tilde{g}^{d_0+l_i}$ on $M^{4k+1}_{d_0+l_i}$ to $M^{4k+1}_{d_0}$, i.e.,
$$
h_{l_i}:=F_{l_i}^*\left(\tilde{g}^{d_0+l_i}\right)\in \mathcal{R}_{Ric>0}(M^{4k+1}_{d_0}).
$$
As in the proof of Theorem A, by looking at the \pullback \ by $F_{l_i}$ of the bundle $\alpha$ and of the preferred $\spinc$-structure on $M^{4k+1}_{d_0+l_i}$ one obtains the following result (using the fact that relative eta-invariants are preserved via \pullback s together with Proposition \ref{PROP: eta-inv of Ric metrics}).

\begin{claim}\label{CLAIM2 for RIC}
There exists an infinite subset $\{l_i'\}\subset\{l_i\}$ with $i\in\N$ and a topological $\spinc$-structure on $M^{4k+1}_{d_0}$ satisfying the following. For $(M^{4k+1}_{d_0} ,h_{l_i'})$, the relative eta-invariant of the corresponding $\spinc$-Dirac operator twisted with $\alpha$ is equal to 
$$
 \widetilde{\eta}_\alpha \left(M^{4k+1}_{d_0},h_{l_i'} \right) =\widetilde{\eta}_\alpha \left(M^{4k+1}_{d_0 + l_i'}, \tilde{g}^{d_0+l_i'}\right) = -2^{-2k}\left( d_0 + l_i' \right).
$$\qed
\end{claim}

Now we can argue as in the proof of Theorem A to distinguish components in $\mathcal{M}_{Ric > 0}(M^{4k+1}_{d_0})$. Note that, in contrast to the proof of Theorem A, there is no need to use the Ricci flow, since the lift of a path in $\mathcal{R}_{Ric > 0}(M^{4k+1}_{d_0})/\mathcal{D}$ already lies in $\mathcal{R}_{scal > 0}(M^{4k+1}_{d_0})$. We get the following result, which implies Theorem B.

\begin{theorem}\label{THM: aux theorem B}
Let $\{l_i'\}_{i\in\N}$ be the infinite subset from Claim \ref{CLAIM2 for RIC}. Then the metrics $\{h_{l_i'}\}_{i\in\N}$ represent infinitely many path components of $\mathcal{M}_{Ric > 0}(M^{4k+1}_{d_0})$.\qed
\end{theorem}


\subsection{Cheeger deformation approach}\label{SS: Speranca approach}

As before let $d,n$ odd and let $d,n\geq 3$. The main step in this proof of Proposition \ref{PROP: eta-inv of Ric metrics} is to construct an $O(n)$-invariant metric $\tilde{g}$ on $M_d^{2n-1}$ with the following property: there exists $t_0\geq 0$ such that the Cheeger deformation $\tilde{g}_{t_0}$ is of $Ric>0$ and $\tilde{g}_t$ is of $scal >0$ for every $t\geq t_0$. Then the relative eta-invariant of $\tilde{g}^d:=\tilde{g}_{t_0}$ can be computed as in the proof of Claim \ref{CLAIM}, which gives us for $2n-1=4k+1$ the desired value $-2^{-(n-1)}d=-2^{-2k}d$ (see Claim \ref{CLAIM: computation of eta-inv for Cheeger scal>0}). The existence of such a metric is guaranteed by the following more general result.

\begin{proposition}\label{PROP: Cheeger def Ric>0 on Brieskorn sphere}
The manifold $\Sigma_0^{2n-1}(d)$ admits an $O(n)$-invariant metric $g$ with the following property: there exists $t_0\geq 0$ such that the Cheeger deformation $g_{t}$ is of $Ric>0$ for every $t\geq t_0$. The quotient $M_d^{2n-1}=\Sigma_0^{2n-1}(d)/\tau$ inherits a metric $\tilde{g}$ with the same properties.
\end{proposition}

\begin{proof} The proof consists of three parts: First we give a different description of $\Sigma_0^{2n-1}(d)$ as a certain attaching space $\hat \Sigma ^{2n-1}_d$. Then we construct a principal $O(n)$-bundle $P_d$ over $\hat \Sigma ^{2n-1}_d$ which is also the total space of another $O(n)$-bundle over the round $(2n-1)$-dimensional sphere. Finally we apply work of Searle and Wilhelm to this situation to construct the corresponding metric.

\smallskip

\noindent\textbf{A different description of $\Sigma_0^{2n-1}(d)$.}
As shown in \cite[V.9]{B72} the Brieskorn sphere $\Sigma_0^{2n-1}(d)$ is $O(n)$-equivariantly diffeomorphic to the following equivariant attaching space (see \cite[I.6.C]{B72} for the definition)
\begin{equation}\label{EQ: usual description of Sigma}
S^{n-1}\times D^n\bigcup_{\varphi^d} S^{n-1}\times D^n.
\end{equation}
Let us explain the notation and the actions involved in \eqref{EQ: usual description of Sigma}. The $O(n)$-action on $D^n$ and $S^{n-1}$ is the standard one from the left. The $O(n)$-action on $O(n)$ is by conjugation. Let $\theta : S^{n-1} \to O(n)$ be the map $x\mapsto \theta_x$, where $\theta_x$ denotes the reflection through the line $\R x$. Note that the map $\theta$ is $O(n)$-equivariant.
The attaching map $\varphi^d$ in \eqref{EQ: usual description of Sigma} is defined via the map $\varphi: S^{n-1}\times S^{n-1}\to  S^{n-1}\times S^{n-1}$, $\varphi(x,y):=(\theta _x(y),x)$. It follows that $\varphi$ is $O(n)$-equivariant as well, and the same holds for its powers $\varphi^l$, $l\in \Z$. In particular, the $O(n)$-action on the two halves of \eqref{EQ: usual description of Sigma} (the left-hand and right-hand side of the union) induce an $O(n)$-action on the attaching space. Observe that $\varphi^{2l}(x,y)=((\theta_x\theta_y)^l(x),(\theta_x\theta_y)^l(y))$.

It will be however more convenient to express $\Sigma_0^{2n-1}(d)$ in a different way other than \eqref{EQ: usual description of Sigma}. Here we follow Bredon's exposition in \cite[I.7]{B72} for the case $d=1$ and generalize it to every odd $d=2l+1$. Consider the equivariant attaching space
\begin{equation}\label{EQ: different description of Sigma}
\hat \Sigma ^{2n-1}_d:=D^{n}\times S^{n-1}\bigcup_{\varphi^{d-1}}S^{n-1}\times D^{n}.
\end{equation}
One can see that the expressions \eqref{EQ: usual description of Sigma} and \eqref{EQ: different description of Sigma} yield $O(n)$-equivariant diffeomorphic spaces as follows. First apply the automorphism $T(x,y):=(y,x)$ of $D^n\times D^n$ to the first half of \eqref{EQ: usual description of Sigma} and compensate the gluing function by composing it with $T^{-1}=T$. One gets the space 
\begin{equation}\label{EQ: intermediate description of Sigma}
D^{n}\times S^{n-1}\bigcup_{\varphi^{d}\circ T}S^{n-1}\times D^{n},
\end{equation}
which is $O(n)$-equivariantly diffeomorphic to the space in \eqref{EQ: usual description of Sigma}.

Now observe that the map $\varphi\circ T:(y,x)\mapsto (\theta _x(y),x)$ extends to a map on $D^n\times S^{n-1}$ and apply it to the first half of \eqref{EQ: intermediate description of Sigma}. Compensate the gluing function $\varphi^{d}\circ T=\varphi^{d-1}\circ\varphi\circ T$ by composing it with $(\varphi\circ T)^{-1}$. Then one gets precisely the expression in \eqref{EQ: different description of Sigma} as indicated in the following diagram:

$$
\begin{CD}D^n\times S^{n-1}\, @.\bigcup_{\varphi^d\circ T}\, S^{n-1}\times D^n\\
@V\varphi \circ T VV @VVIdV\\
D^n\times S^{n-1}\, @.\bigcup_{\varphi^{d-1}}\, S^{n-1}\times D^n.\\
\end{CD}
$$

\smallskip

\noindent\textbf{An $O(n)$-$O(n)$ bundle $P_d$.} Consider again the manifold $\hat \Sigma^{2n-1}_d$ in \eqref{EQ: different description of Sigma} with $d=2l+1$. We lift the gluing function $\varphi^{2l}$ to a clutching function $(\varphi^{2l},(\theta_x\theta_y)^l):S^{n-1}\times S^{n-1}\times O(n)\to  S^{n-1}\times S^{n-1}\times O(n)$ for a principal $O(n)$-bundle over $\hat \Sigma^{2n-1}_d$, where $(\theta_x\theta_y)^l$ denotes the map $O(n)\to O(n)$ defined by left matrix multiplication. Let $P_d$ be defined as the attaching space
$$
P_d:= D^{n}\times S^{n-1}\times O(n)\bigcup_{(\varphi^{2l},(\theta_x\theta_y)^l)}S^{n-1}\times D^{n}\times O(n).
$$
By construction, $P_d\to \hat \Sigma^{2n-1}_d$ is an $O(n)$-principal bundle with principal action, denoted by $\bullet$, given by right multiplication of $O(n)$ on itself in each half, i.e., $(x,y,B)\bullet A:=(x,y,BA)$.

The crucial fact of this construction is that $P_d$ has a second free (left) action by $O(n)$, denoted by $\star$. It is determined by $A\star(x,y,B)=(Ax,Ay,AB)$ on each half. Note that this defines an action on $P_d$ since the map $(\varphi^{2l},(\theta_x\theta_y)^l)$ is $O(n)$-equivariant. Observe that the actions $\bullet$ and $\star$ on $P_d$ commute.

The quotient of $P_d$ under the $\star$ action is diffeomorphic to $S^{2n-1}$ as we shall explain. First note that the $\star$-orbit of $(x,y,B)$ can be identified with $(B^{-1}x,B^{-1}y)$ and hence the quotient space of $P_d$ under $\star$ equals $D^n\times S^{n-1}\cup S^{n-1}\times D^n$ for some gluing function. To determine the gluing function let us look at a point $(x,y)$ in the (boundary of the) first half. Now $(x,y)$ corresponds to the $\star$-orbit of $(x,y,Id)$ which is mapped under the attaching map to the  $\star$-orbit of $(\varphi^{2l},(\theta_x\theta_y)^l)(x,y,Id)= ((\theta_x\theta_y)^l (x),(\theta_x\theta_y)^l (y),(\theta_x\theta_y)^l Id)$. The latter corresponds to $(x,y)$ in the (boundary of the) second half. Thus the gluing map is the identity and the resulting manifold 
\begin{equation}\label{EQ: sphere gluing}
D^n\times S^{n-1}\bigcup_{Id} S^{n-1}\times D^n
\end{equation}
is diffeomorphic to $S^{2n-1}$ via $F: (x,y)\mapsto (x,y)/\vert (x,y)\vert$,
where $\vert (x,y)\vert$ denotes the Euclidean norm in $\R^{2n}$.

By the commutativity of the $O(n)$-actions on $P_d$, the spaces $\hat \Sigma^{2n-1}_d$ and $S^{2n-1}$ inherit an $O(n)$-action. Moreover, the quotient spaces $\hat \Sigma^{2n-1}_d /O(n)$ and $S^{2n-1}/O(n)$ are canonically identified. In the terminology of Speran\c{c}a $P_d$ is an $O(n)$-$O(n)$-manifold, in \cite[\S 5]{S16} he explains the construction of $P_d$ for the case $d=3$.

We will need below the following fact: the induced (right) $O(n)$-action on $S^{2n-1}$ is linear. In fact, $A\in O(n)$ maps $(a,b)\in S^{2n-1}\subset \R ^n \times \R ^n$ to $(A^{-1}a,A^{-1}b)$. To see this, consider a point $(x,y)$ in one half of \eqref{EQ: sphere gluing} and identify it with the $\star$-orbit of $(x,y,Id)$ in the corresponding half of $P_d$. Now $A$ maps the $\star$-orbit of $(x,y,Id)$ to the $\star$-orbit of $(x,y,Id)\bullet A = (x,y,A)$. Thus the induced action of $A\in O(n)$ on each half of \eqref{EQ: sphere gluing} is of the form $(x,y)A=(A^{-1}x,A^{-1}y)$ and the induced action on $S^{2n-1}$ \wrt \ the diffeomorphism $F$ is given by the linear action $(a,b)\mapsto (A^{-1}a,A^{-1}b)$.

\smallskip

\noindent\textbf{Construction of the metrics.} Next we construct metrics on the manifolds above. First we use a construction of Cavenaghi and Speran\c{c}a to endow $G$-$G$ bundles with compatible invariant metrics \cite[Cor. 5.2]{CS17}. In our situation their work implies the following: given an $O(n)$-invariant metric on $S^{2n-1}$, one can construct an $O(n)$-invariant metric on $\hat \Sigma^{2n-1}_d$ such that the quotients $\hat \Sigma^{2n-1}_d /O(n)$ and $S^{2n-1}/O(n)$ are isometric as metric spaces. For curvature-related reasons it will be convenient to consider the round metric on $S^{2n-1}$, which is $O(n)$-invariant since the action is linear. Denote by $h$ the corresponding $O(n)$-invariant metric on $\hat \Sigma^{2n-1}_d$. Observe that the regular part of the quotient space $S^{2n-1}/O(n)=\hat \Sigma^{2n-1}_d /O(n)$ inherits a metric which is of positive sectional curvature due to the Gray-O'Neill formula \cite{ON66,G67} for Riemannian submersions.

Next we use a construction of Searle and Wilhelm for metrics of $Ric> 0$ on manifolds with group actions \cite{SW15}. They consider the following situation. Let $G$ be a compact Lie group acting by isometries on a closed Riemannian manifold $(M,h)$ such that the fundamental group of the principal orbit is finite, and the induced metric on the regular part of the quotient $M/G$ has $Ric >0$. Then they deform the metric $h$ in two steps: firstly, they do a certain conformal change, yielding a $G$-invariant metric $\overline{h}$; secondly, they do a Cheeger deformation $\overline{h}_t$ of $\overline{h}$ using the $G$-action. They show that, for a certain $t_0\geq 0$, the metric $\overline{h}_t$ is of $Ric>0$ for every $t\geq t_0$.

Note that the construction of Searle and Wilhelm applies to our $O(n)$-action on $(\hat \Sigma^{2n-1}_d,h)$, where $h$ is the \pullback \ of the round metric on $S^{2n-1}$. As discussed above the quotient space $\hat \Sigma^{2n-1}_d/O(n)$ has in particular $Ric>0$ in the regular part. The principal orbit has finite fundamental group. This follows from the fact that $\hat \Sigma^{2n-1}_d$ is $O(n)$-equivariantly diffeomorphic to $\Sigma_0^{2n-1}(d)$ and hence the corresponding principal orbits of each action are diffeomorphic (see Section \ref{SS: group actions and quotients}). Thus, there is an $O(n)$-invariant metric $\overline{h}$ on $\hat \Sigma^{2n-1}_d$ such that, for some $t_0\geq 0$, the Cheeger deformation $\overline{h}_t$ is of $Ric>0$ for every $t\geq t_0$ (an alternative argument which yields a metric of $Ric>0$ only using Cheeger deformation can be found in \cite[Thm. 6.3, Thm. 6.6]{CS17}). We set $g$ as the \pullback \ of $\overline{h}$ to $\Sigma_0^{2n-1}(d)$, and this metric satisfies the desired properties. Finally, since $g$ is $O(n)$-invariant it descends to an $O(n)$-invariant metric $\tilde{g}$ on the quotient $M_d^{2n-1}=\Sigma_0^{2n-1}(d)/\tau $ with the stated properties.
\end{proof}

\subsection{Surgery and plumbing approach}\label{SS: Wraith approach}

As before let $d,n$ odd and let $d,n\geq 3$ and $4k+1=2n-1$. To prove Proposition \ref{PROP: eta-inv of Ric metrics} we first review a different description of $\Sigma_0 ^{4k+1}(d)$ which was used by Wraith to construct metrics of $Ric > 0$. We then observe that Wraith's construction can be refined, allowing us to compute the relative eta-invariant to be equal to $-2^{-2k}d$.

As noted by Hirzebruch the homotopy sphere $\Sigma_0 ^{4k+1}(d)$ is as an $O(2k+1)$-manifold diffeomorphic to the boundary of the manifold obtained by plumbing $(d-1)$ disk tangent bundles \wrt \ the graph $A_{d-1}$. We briefly recall the construction and refer the reader to  \cite{HM68,B72} for more details.

Let $D^m$ and $S^{m-1}$ denote the unit disk and unit sphere in $\R ^m$, respectively. The standard action of $O(2k+1)\hookrightarrow O(2k+2)$ on $\R ^{2k+2}$ restricts to an action on $S^{2k+1}$ with precisely two fixed points. Let $D(TS^{2k+1})$ denote the unit disk tangent bundle. Consider the induced $O(2k+1)$-action on $D(TS^{2k+1})$. Let $U\subset S^{2k+1}$ be a small closed neighborhood of a fixed point which is equivariantly diffeomorphic to $D^{2k+1}$. We identify the restriction of the disk bundle $D(TS^{2k+1})$ to $U$ equivariantly with the trivial disk bundle $D^{2k+1}\times D(\R ^{2k+1})$, where $D(\R ^{2k+1})$ also denotes the unit disk in $\R ^{2k+1}$ and $O(2k+1)$ acts diagonally on the two factors. Note that the involution $\tau := -Id\in O(2k+1)$ acts by $(-Id,-Id)$ on $D^{2k+1}\times D(\R ^{2k+1})$.

Let $A_{d-1}$ denote the line graph $\dynkin{A}{}$ with $(d-1)$ vertices. Now consider the plumbing of $(d-1)$ copies of $D(TS^{2k+1})$ at fixed points of the $O(2k+1)$-action \wrt \ the graph $A_{d-1}$. Let $W(d)$ denote the $O(2k+1)$-manifold obtained by this plumbing. Its boundary $\partial W(d)$ is equivariantly diffeomorphic to $\Sigma_0 ^{4k+1}(d)$ (see \cite[p. 77]{HM68} and \cite[Thm. V.8.1, Thm. V.9.2]{B72}). We fix such an identification. Note that the involution $\tau := -Id\in O(2k+1)$ acts without fixed points on $\partial W(d)$ and with $d$ isolated fixed points on $W(d)$. With respect to the identification above the quotient $\partial W(d)/\tau $ is given by $M^{4k+1}_d$.

As in the case of Brieskorn spheres the preferred $\spinc$-structure of $\partial W(d)/\tau $ is induced by a $\tau$-equivariant almost complex structure on $W(d)$ as we will explain next. First recall that for any smooth manifold $N$ with tangent bundle $\pi :TN\to N$ a connection for $TN$ gives a splitting $T(TN)\cong \pi ^*(TN)\oplus \pi ^*(TN)$ into vertical and horizontal subbundles and induces almost complex structures $\pm J$ on $TN$, where $J$ acts on  $\pi ^*(TN)\oplus \pi ^*(TN)$ by $(v,w)\mapsto (-w,v)$. Recall also that if the connection is equivariant \wrt \ a diffeomorphism $\varphi :N\to N$ then the induced action of $\varphi $ on $TN$ will be compatible with the almost complex structure, i.e., $ J$ is $\varphi $-equivariant.

To define the almost complex structure on $W(d)$ we first consider the building block $D(TS^{2k+1})$. Near the two $\tau$-fixed points of $S^{2k+1}$ we identify the disk bundle $D(TS^{2k+1})$ equivariantly with $D^{2k+1}\times D(\R ^{2k+1})$, as before. We fix a $\tau$-equivariant connection for $TS^{2k+1}$ which is trivial \wrt \ the trivializations.

Let us call $W_i$ the $i$th step in the plumbing process starting with $W_1=D(TS^{2k+1})$ and ending with $W_{d-1}=W(d)$. On the tangent bundle $TW_1$ we choose one of the complex structures, say $J_1$, induced by the connection on $TS^{2k+1}$. By construction $J_1$ is $\tau$-equivariant. Suppose the $\tau$-equivariant complex structure $J_i$ on $TW_i$ has already been constructed and $W_{i+1}$ is obtained by plumbing $W_i$ and $D(TS^{2k+1})$ at $\tau $-fixed points. Note that $J_i$ determines a complex structure $J_{i+1}$ for $TW_{i+1}$ which is compatible with $J_i$ near the fixed point \wrt \ plumbing and restricts to one of the complex structures induced by the connection for $TS^{2k+1}$. Note also that $J_{i+1}$ is $\tau$-equivariant. Since $A_{d-1}$ has no cycles we obtain a well-defined almost complex structure $J_W$ on $W(d)$ which is preserved under the action of the involution $\tau$. As before $J_W$, together with a fixed $\tau$-equivariant Hermitian metric on $TW$, induces a $\spinc$-structure on the boundary $\partial W(d)$ which yields the preferred $\spinc$-structure on the quotient $\partial W(d)/\tau \cong M^{4k+1}_d$.

The metric $\tilde{g}^d$ on $M^{4k+1}_d$ which we use in this proof of Proposition \ref{PROP: eta-inv of Ric metrics} is based on a construction of Wraith. In \cite{W97} Wraith constructed metrics of $Ric>0$ on $\partial W(d)$ and on the boundary of many other plumbing constructions. In \cite{W11} he proved that the metric of $Ric>0$ can be chosen in such a way that after a deformation via metrics of $scal>0$ the deformed metric can be extended to a metric of $scal>0$ on $W(d)$ which is of product form near the boundary. Based on this construction Wraith showed that for every $(4k-1)$-dimensional homotopy sphere which bounds a parallelizable manifold the moduli space of metrics of $Ric > 0$ has infinitely many path components.
 
An inspection of Wraith's construction in \cite{W11} and the references therein shows that if one does the plumbing construction for $W(d)$ at $\tau$-fixed points \wrt \ the graph $A_{d-1}$ one can ensure also that the metrics are $\tau$-invariant. The construction of metrics and their deformations are done step by step in the surgery/plumbing process. In each step one has to choose connections for the respective sphere and disk bundles which depend on the trivializations of these bundles. In the situation considered above the trivializations are $\tau$-equivariant which implies that the involved connections can be assumed to be $\tau$-equivariant. All other ingredients in the surgery/plumbing process depend on functions of a radial parameter (distance to the $\tau$-fixed point) which are clearly $\tau$-equivariant.

The upshot is that Wraith's construction applied to the plumbing of $(d-1)$ copies of $D(TS^{2k+1})$ at $\tau$-fixed points \wrt \ the graph $A_{d-1}$ gives a $\tau $-invariant metric $g^d$ of $Ric>0$ on the boundary $\partial W(d)$ such that after a deformation via $\tau $-invariant metrics of $scal>0$ the deformed metric can be extended to a $\tau $-invariant metric $g_W$ of $scal>0$ on $W(d)$ which is of product form near the boundary.

Let $\tilde{g}^d$ and $\tilde {h}^d$ denote the metrics on $M^{4k+1}_d$ induced by $g^d$ and the restriction $g_W\vert _{\partial W(d)}$, respectively, under the identification $\partial W(d)/\tau \cong M^{4k+1}_d$.

We are now in the position to prove Proposition \ref{PROP: eta-inv of Ric metrics}. We equip $M^{4k+1}_d$ with the preferred $\spinc $-structure induced from the almost complex structure on $W(d)$ and the metrics $\tilde{g}^d$ (resp. $\tilde {h}^d$) and fix a flat connection on the associated $U(1)$-principal bundle. Let $\alpha $ denote the non-trivial flat complex line bundle over $M^{4k+1}_d$. Consider the relative eta-invariants $\widetilde{\eta}_\alpha \left(M^{4k+1}_d,\tilde{g}^d\right)$  and $\widetilde{\eta}_\alpha \left(M^{4k+1}_d,\tilde{h}^d\right)$ for the corresponding $\spinc $-Dirac operators twisted with $\alpha $.  Since $\tilde{g}^d$ and $\tilde {h}^d$ are connected by a path of metrics of $scal>0$ one has by Proposition \ref{PROP: eta invariant scal} that $\widetilde{\eta}_\alpha \left(M^{4k+1}_d,\tilde{g}^d\right)=\widetilde{\eta}_\alpha \left(M^{4k+1}_d,\tilde{h}^d\right)$. Hence, it suffices to compute $\widetilde{\eta}_\alpha \left(M^{4k+1}_d,\tilde{h}^d\right)$.

Using Donnelly's fixed point formula $\widetilde{\eta}_\alpha \left(M^{4k+1}_d,\tilde{h}^d\right)$ can be computed as the sum of local contributions at the $\tau$-fixed points in $W(d)$. Since the $\spinc $-structure is induced from a $\tau$-equivariant almost complex structure one finds that each contribution is equal to $\frac 1 {2^{2k+1}}$ (see Proposition \ref{PROP: Dolbeault contribution}). Arguing as in the proof of Proposition \ref{PROP: eta of special metrics with s nonzero} it follows that
$$\widetilde{\eta}_\alpha \left(M^{4k+1}_d,\tilde{g}^d\right)=\widetilde{\eta}_\alpha \left(M^{4k+1}_d,\tilde{h}^d\right)=-2\cdot d \cdot \frac 1 {2^{2k+1}}=- \frac d {2^{2k}}.$$\qed

\bigskip
\noindent
\textsc{Department of Mathematics, University of Fribourg, Switzerland}\\
{\em E-mail address:} \textrm{anand.dessai@unifr.ch}\\

\smallskip
\noindent
\textsc{ETSI de Caminos, Canales y Puertos, Universidad Polit\'ecnica de Madrid, Spain}\\
{\em E-mail address:} \textrm{david.gonzalez.alvaro@upm.es}


{\small 
\begin{thebibliography}{9}
\bibitem[AB15]{AB15} Alexandrino, Marcos M.; Bettiol, Renato G. \emph{Lie groups and geometric aspects of isometric actions}. Springer, Cham, 2015.

\bibitem[AB68]{AB68} Atiyah, M. F.; Bott, R. \emph{A Lefschetz fixed point formula for elliptic complexes. II}. Applications. Ann. of Math. (2) 88 1968 451--491. 

\bibitem[A77]{A77} Atiyah, M. \emph{Note on involutions}. \cite {KS77}, p. 338

\bibitem[ABS64]{ABS64} Atiyah, M. F.; Bott, R.; Shapiro, A. \emph{Clifford modules}. Topology 3 1964 suppl. 1 3--38. 

\bibitem[APS73]{APS73} Atiyah, M. F.; Patodi, V. K.; Singer, I. M. \emph{Spectral asymmetry and Riemannian geometry}. Bull. London Math. Soc. 5 (1973), 229--234.

\bibitem[APSI75]{APSI75} Atiyah, M. F.; Patodi, V. K.; Singer, I. M. \emph{Spectral asymmetry and Riemannian geometry. I}. Math. Proc. Cambridge Philos. Soc. 77 (1975), 43--69. 

\bibitem[APSII75]{APSII75} Atiyah, M. F.; Patodi, V. K.; Singer, I. M. \emph{Spectral asymmetry and Riemannian geometry. II}. Math. Proc. Cambridge Philos. Soc. 78 (1975), no. 3, 405--432. 

\bibitem[ASIII68]{ASIII68} Atiyah, M. F.; Singer, I. M. \emph{The index of elliptic operators. III}. Ann. of Math. (2) 87 1968 546--604. 

\bibitem[BH87]{BH87} Back, Allen; Hsiang, Wu-Yi. \emph{Equivariant geometry and Kervaire spheres}. Trans. Amer. Math. Soc. 304 (1987), no. 1, 207--227. 

\bibitem[BKS11]{BKS11} Belegradek, Igor; Kwasik, Slawomir; Schultz, Reinhard. \emph{Moduli spaces of nonnegative sectional curvature and non-unique souls}. J. Differential Geom. 89 (2011), no. 1, 49--85. 

\bibitem[BW07]{BW07} B\"ohm, Christoph; Wilking, Burkhard. \emph{Nonnegatively curved manifolds with finite fundamental groups admit metrics with positive Ricci curvature}. Geom. Funct. Anal. 17 (2007), no. 3, 665--681.

\bibitem[BG95]{BG95} Botvinnik, Boris; Gilkey, Peter B. \emph{The eta invariant and metrics of positive scalar curvature}. Math. Ann. 302 (1995), no. 3, 507--517.

\bibitem[BG96]{BG96} Botvinnik, Boris; Gilkey, Peter B. \emph{Metrics of positive scalar curvature on spherical space forms}. Canad. J. Math. 48 (1996), no. 1, 64--80. 

\bibitem[BGS97]{BGS97} Botvinnik, Boris; Gilkey, Peter; Stolz, Stephan. \emph{The Gromov-Lawson-Rosenberg conjecture for groups with periodic cohomology}. J. Differential Geom. 46 (1997), no. 3, 374--405.

\bibitem[B75]{B75} Bourguignon, Jean-Pierre. \emph{Une stratification de l'espace des structures riemanniennes}. (French) Compositio Math. 30 (1975), 1--41.

\bibitem[BGN03]{BGN03} Boyer, Charles P.; Galicki, Krzysztof; Nakamaye, Michael. \emph{Sasakian geometry, homotopy spheres and positive Ricci curvature}. Topology 42 (2003), no. 5, 981--1002.

\bibitem[B72]{B72} Bredon, Glen E. \emph{Introduction to compact transformation groups}. Pure and Applied Mathematics, Vol. 46. Academic Press, New York-London, 1972.

\bibitem[B66]{B66} Brieskorn, Egbert. \emph{Beispiele zur Differentialtopologie von Singularit\"aten}. (German) Invent. Math. 2 1966 1--14.

\bibitem[CS17]{CS17} Cavenaghi, Leonardo F.; Speran\c{c}a, Llohann D. \emph{On the geometry of some equivariantly related manifolds}, preprint arXiv:1708.07541 (2017), to appear in IMRN.

\bibitem[C72]{C72} Cheeger, Jeff. \emph{Some examples of manifolds of nonnegative curvature}. J. Differential Geometry 8 (1972), 623--628.

\bibitem[De96]{De96} Dessai, Anand. \emph{Rigidity theorems for $\spinc$-manifolds and applications}, Dissertation, Mainz (1996), published by Shaker Verlag (1997).

\bibitem[DKT18]{DKT18} Dessai, Anand; Klaus, Stephan; Tuschmann, Wilderich. \emph{Nonconnected moduli spaces of nonnegative sectional curvature metrics on simply connected manifolds}. Bull. Lond. Math. Soc. 50 (2018), no. 1, 96--107.

\bibitem[D17]{D17} Dessai, Anand. \emph{On the moduli space of nonnegatively curved metrics on Milnor spheres}. Preprint, arXiv:1712.08821 (2017).

\bibitem[D78]{D78} Donnelly, Harold. \emph{Eta invariants for $G$-spaces}. Indiana Univ. Math. J. 27 (1978), no. 6, 889--918.

\bibitem[Du96]{Du96} Duistermaat, J. J. \emph{The heat kernel Lefschetz fixed point formula for the spin-$c$ Dirac operator}. Reprint of the 1996 edition. Modern Birkh\"auser Classics. Birkh\"auser/Springer, New York, 2011.

\bibitem[E70]{E70} Ebin, David G. \emph{The manifold of Riemannian metrics}. 1970 Global Analysis (Proc. Sympos. Pure Math., Vol. XV, Berkeley, Calif., 1968) pp. 11--40 Amer. Math. Soc., Providence, R.I. 

\bibitem[G69]{G69} Giffen, Charles H. \emph{Smooth homotopy projective spaces}. Bull. Amer. Math. Soc. 75 1969 509--513. 

\bibitem[G95]{G95} Gilkey, Peter B. \emph{Invariance theory, the heat equation, and the Atiyah-Singer index theorem}. Second edition. Studies in Advanced Mathematics. CRC Press, Boca Raton, FL, 1995. 

\bibitem[G67]{G67} Gray, Alfred. \emph{Pseudo-Riemannian almost product manifolds and submersions}. J. Math. Mech. 16 1967 715--737.

\bibitem[G17]{G17} Goodman, McFeely Jackson. \emph{On the moduli space of metrics with nonnegative sectional curvature}. Preprint, arXiv:1712.01107 (2017).

\bibitem[GL80]{GL80} Gromov, Mikhael; Lawson, H. Blaine, Jr. \emph{The classification of simply connected manifolds of positive scalar curvature}. Ann. of Math. (2) 111 (1980), no. 3, 423--434.

\bibitem[GVWZ06]{GVWZ06} Grove, Karsten; Verdiani, Luigi; Wilking, Burkhard; Ziller, Wolfgang. \emph{Non-negative curvature obstructions in cohomogeneity one and the Kervaire spheres}. Ann. Sc. Norm. Super. Pisa Cl. Sci. (5) 5 (2006), no. 2, 159--170.

\bibitem[GZ02]{GZ02} Grove, Karsten; Ziller, Wolfgang. \emph{Cohomogeneity one manifolds with positive Ricci curvature}. Invent. Math. 149 (2002), no. 3, 619--646.

\bibitem[GZ00]{GZ00} Grove, Karsten; Ziller, Wolfgang. \emph{Curvature and symmetry of Milnor spheres}. Ann. of Math. (2) 152 (2000), no. 1, 331--367.

\bibitem[GGK02]{GGK02} Guillemin, Victor; Ginzburg, Viktor; Karshon, Yael. \emph{Moment maps, cobordisms, and Hamiltonian group actions}. Appendix J by Maxim Braverman. Mathematical Surveys and Monographs, 98. American Mathematical Society, Providence, RI, 2002.

\bibitem[H75]{H75} Hern\'andez-Andrade, Horacio. \emph{A class of compact manifolds with positive Ricci curvature}. Differential geometry (Proc. Sympos. Pure Math., Vol. XXVII, Stanford Univ., Stanford, Calif., 1973), Part 1, pp. 73--87. Amer. Math. Soc., Providence, R.I., 1975.

\bibitem[HHR06]{HHR06} Hill, M. A.; Hopkins, M. J.; Ravenel, D. C. \emph{On the nonexistence of elements of Kervaire invariant one}. Ann. of Math. (2) 184 (2016), no. 1, 1--262.

\bibitem[HM68]{HM68} Hirzebruch, F.; Mayer, K. H. \emph{${\rm O}(n)$-Mannigfaltigkeiten, exotische Sph\"aren und Singularit\"aten}. (German) Lecture Notes in Mathematics, No. 57 Springer-Verlag, Berlin-New York 1968.

\bibitem[HBJ92]{HBJ92} Hirzebruch, Friedrich; Berger, Thomas; Jung, Rainer. \emph{Manifolds and modular forms}. With appendices by Nils-Peter Skoruppa and by Paul Baum. Aspects of Mathematics, E20. Friedr. Vieweg \& Sohn, Braunschweig, 1992.

\bibitem[KPT05]{KPT05} Kapovitch, Vitali; Petrunin, Anton; Tuschmann, Wilderich. \emph{Non-negative pinching, moduli spaces and bundles with infinitely many souls}. J. Differential Geom. 71 (2005), no. 3, 365--383.

\bibitem[KS77]{KS77} Kirby, Robion C.; Siebenmann, Laurence C.  \emph{Foundational Essays on Topological Manifolds, Smoothings, and Triangulations}. Princeton Univ. Press (1977)


\bibitem[KS93]{KS93} Kreck, Matthias; Stolz, Stephan. \emph{Nonconnected moduli spaces of positive sectional curvature metrics}. J. Amer. Math. Soc. 6 (1993), no. 4, 825--850.

\bibitem[LM89]{LM89} Lawson, H. Blaine, Jr.; Michelsohn, Marie-Louise. \emph{Spin geometry}. Princeton Mathematical Series, 38. Princeton University Press, Princeton, NJ, 1989.

\bibitem[LY74]{LY74} Lawson, H. Blaine, Jr.; Yau, Shing Tung. \emph{Scalar curvature, non-abelian group actions, and the degree of symmetry of exotic spheres}. Comment. Math. Helv. 49 (1974), 232--244.

\bibitem[L63]{L63} Lichnerowicz, André. \emph{Spineurs harmoniques}. (French) C. R. Acad. Sci. Paris 257 1963 7--9.

\bibitem[L71]{L71} L\'opez de Medrano, S. \emph{Involutions on manifolds}. Ergebnisse der Mathematik und ihrer Grenzgebiete, Band 59. Springer-Verlag, New York-Heidelberg, 1971.

\bibitem[M68]{M68} Milnor, John. \emph{Singular points of complex hypersurfaces}. Annals of Mathematics Studies, No. 61 Princeton University Press, Princeton, N.J.; University of Tokyo Press, Tokyo 1968.

\bibitem[M87]{M87} M\"uter, Michael. \emph{Kr\"ummungserh\"ohende Deformationen mittels Gruppenaktionen}.  Ph.D. thesis (M\"unster) 1987.

\bibitem[N07]{N07} Nicolaescu, Liviu I. \emph{Lectures on the geometry of manifolds}. Second edition. World Scientific Publishing Co. Pte. Ltd., Hackensack, NJ, 2007.

\bibitem[ON66]{ON66} O'Neill, Barrett. \emph{The fundamental equations of a submersion}. Michigan Math. J. 13 1966 459--469.


\bibitem[P81]{P81} Poor, Walter A. \emph{Differential geometric structures}. McGraw-Hill Book Co., New York, 1981.

\bibitem[ST04]{ST04} Schwachh\"ofer, Lorenz J.; Tuschmann, Wilderich. \emph{Metrics of positive Ricci curvature on quotient spaces}. Math. Ann. 330 (2004), no. 1, 59--91.

\bibitem[SW15]{SW15} Searle, Catherine; Wilhelm, Frederick. \emph{How to lift positive Ricci curvature}. Geom. Topol. 19 (2015), no. 3, 1409--1475.

\bibitem[S61]{S61} Smale, Stephen. \emph{Generalized Poincar\'e's conjecture in dimensions greater than four}. Ann. of Math. (2) 74 1961 391--406.

\bibitem[S16]{S16} Speran\c{c}a, L. D. \emph{Pulling back the Gromoll-Meyer construction and models of exotic spheres}. Proc. Amer. Math. Soc. 144 (2016), no. 7, 3181--3196.


\bibitem[T16]{T16} Tuschmann, Wilderich. \emph{Spaces and moduli spaces of Riemannian metrics}. Front. Math. China 11 (2016), no. 5, 1335--1343.

\bibitem[TW17]{TW17} Tuschmann, Wilderich; Wiemeler, Michael. \emph{On the topology of moduli spaces of non-negatively curved Riemannian metrics}. Preprint, arXiv:1712.07052 (2017).

\bibitem[TW15]{TW15} Tuschmann, Wilderich; Wraith, David J. \emph{Moduli spaces of Riemannian metrics}. Second corrected printing. Oberwolfach Seminars, 46. Birkh\"auser Verlag, Basel, 2015.


\bibitem[W16]{W16} Wiemeler, Michael, \emph{On moduli spaces of positive scalar curvature metrics on highly connected manifolds}, preprint, arXiv: 1610.09658 (2016). 

\bibitem[W97]{W97} Wraith, David J. \emph{Exotic spheres with positive Ricci curvature}. J. Differential Geom. 45 (1997), no. 3, 638--649.

\bibitem[W11]{W11} Wraith, David J. \emph{On the moduli space of positive Ricci curvature metrics on homotopy spheres}. Geom. Topol. 15 (2011), no. 4, 1983--2015.

\bibitem[Z09]{Z09} Ziller, Wolfgang. \emph{On M. Mueter's Ph.D. Thesis on Cheeger deformations}, Lectures by Wolfgang Ziller, preprint, arXiv:0909.0161 (2009).

\end{thebibliography}
}
\end{document}